\documentclass[11pt]{article}
\usepackage{lmodern}

\usepackage[T1]{fontenc}
\usepackage[utf8]{inputenc}
 
 \usepackage{natbib}
  \usepackage{hyperref}
 
\usepackage{fullpage}

\usepackage{color}
\definecolor{lgray}{rgb}{0.95,0.95,0.95}
\definecolor{yel}{rgb}{1,0.98,0.92}
\definecolor{mydarkblue}{rgb}{0,0.0,0.8}
\newcommand\bl[1]{{\color{mydarkblue}#1}}
\newcommand\ff{\bl{f}}

\newcommand\ffc{\bl{f^*}}
\definecolor{mydarkred}{rgb}{0.65,0.0,0.0}
\newcommand\dr[1]{{\color{mydarkred}#1}}
\newcommand\g{\dg{g}}
\newcommand\gc{\dg{g^*}}
\definecolor{mydarkgreen}{rgb}{0,0.47,0}
\newcommand\dg[1]{{\color{mydarkgreen}#1}}
\newcommand\h{\dr{h}}
\newcommand\hc{\dr{h^*}}

\newcommand\K{\mathsf{K}}
\newcommand\Q{\mathsf{Q}}
\newcommand\ze{\mathsf{0}}
\newcommand\Id{\mathsf{I}}
\newcommand{\sqn}[1]{\left\| #1 \right\|^2}
\newcommand\eqdef{\coloneqq}

\usepackage{amsmath, amsfonts,amssymb}
\usepackage{cite}
\DeclareMathOperator*{\argmin}{arg\,min}
\DeclareMathOperator*{\minimize}{minimize}

\usepackage{xspace}
\usepackage{scalefnt}
\definecolor{mydarkredd}{rgb}{0.45,0.16,0.0}
\newcommand{\algn}[1]{{\sf\color{mydarkredd}\scalefont{0.97}{#1}}\xspace}

\usepackage{algorithm,algorithmic}

\usepackage{amsthm}
\theoremstyle{plain}
\newtheorem{theorem}{Theorem}[section]
\newtheorem{proposition}[theorem]{Proposition}
\newtheorem{lemma}[theorem]{Lemma}
\newtheorem{corollary}[theorem]{Corollary}
\theoremstyle{definition}

\theoremstyle{remark}
\newtheorem{remark}[theorem]{Remark}

\usepackage{authblk}

\title{\textbf{A Nesterov-Accelerated Primal--Dual Splitting Algorithm for Convex Nonsmooth Optimization}}

\date{April 10, 2026; revised October 7, 2026}
\author{Laurent Condat\thanks{Corresponding author. Contact: see https://lcondat.github.io/}}
\author{Abdurakhmon Sadiev}
\author{Peter Richt\'arik}
\affil{King Abdullah University of Science and Technology (KAUST),\\ Thuwal, Kingdom of Saudi Arabia}

\usepackage{empheq}

\newcommand*\mycolouredbox[1]{%
\setlength{\fboxsep}{3pt}\colorbox{yel}{\ #1\ }}
{\endEmphEqMainEnv}

\begin{document}
\maketitle

\begin{abstract}
We investigate the integration of Nesterov-type acceleration into primal--dual methods for structured convex optimization. While proximal splitting algorithms efficiently handle composite problems of the form $\min_x f(x) + g(x) + h(\K x)$, accelerating their convergence with respect to the smooth term $f$ is notoriously challenging, due to the rotational dynamics in the primal--dual space. Our guiding principle is that Nesterov acceleration exploits the strong convexity of the dual problem. We first revisit accelerated proximal gradient descent and give a unified analysis of its convex and strongly convex regimes, including the weak convergence of its iterates. We then propose the Accelerated Proximal Alternating Predictor--Corrector algorithm (APAPC) for the setting where $g(x) = \frac{\mu_g}{2}\|x\|^2$. APAPC is fully split: it only calls $\nabla f$, $\K$, $\K^*$ and the proximity operator of $h$, and never solves a linear system. From a single Lyapunov inequality, we derive optimal $\mathcal{O}(1/t^2)$ convergence rates, as well as accelerated linear rates when $\mu_g > 0$, in three regimes where the dual problem is strongly convex: when $h$ is smooth, when $\K^*$ is bounded below, and for linearly constrained problems. In the last two regimes, the $\mathcal{O}(1/t^2)$ rates without strong convexity of the primal problem are, to our knowledge, the first for fully split algorithms. Finally, leveraging recent results on accelerated gradient descent, we establish the weak convergence of the primal iterates and the strong convergence of the dual iterates to a saddle point.
\end{abstract}

\section{Introduction}

Let $\mathcal{X}$ and $\mathcal{U}$ be real Hilbert spaces. We consider structured 
convex optimization problems of the form
\begin{equation}
\minimize_{x\in\mathcal{X}} \;\Psi(x)\coloneqq \ff(x)+\g(x)+\h(\K x), \label{eq1}
\end{equation}
where $\ff:\mathcal{X}\rightarrow \mathbb{R}$ is convex and differentiable, $\g:\mathcal{X}\rightarrow \mathbb{R}\cup\{+\infty\}$ and $\h:\mathcal{U}\rightarrow \mathbb{R}\cup\{+\infty\}$ are proper lower semicontinuous convex functions, and $\K:\mathcal{X}\rightarrow \mathcal{U}$ is a nonzero bounded linear operator \citep{bau17}. 
We assume that  $\ff$ is 
$L_\ff$-smooth for some $L_\ff>0$, that is, its gradient $\nabla \ff$ is $L_\ff$-Lipschitz continuous on $\mathcal{X}$. 

Associated with \eqref{eq1}, we consider the dual problem \citep[Section~15.3]{bau17}
\begin{equation}
\minimize_{u\in\mathcal{U}} \; 
\left(\ff+\g\right)^*(-\K^*u)+\hc(u), \label{eq1cda}
\end{equation}
where $\phi^*$ denotes the convex conjugate of a function $\phi$ \citep[Chapter 13]{bau17}, and $\K^*$ is the adjoint 
of $\K$. For any convex function $\phi$, we denote by $\mu_\phi\geq 0$ a constant such that $\phi$ is $\mu_\phi$-strongly convex,  that is, $\phi-\frac{\mu_\phi}{2}\|\cdot\|^2$ is convex.

Problems of the form \eqref{eq1} arise in signal and image processing, 
inverse problems, control, and machine learning, where $\g$ and $\h$ encode nonsmooth regularization terms and constraints 
\citep{glo16, cha16, bot18}.

We recall that, for a proper lower semicontinuous 
convex function $\phi$ on $\mathcal{X}$, 
its proximity operator is $\mathrm{prox}_\phi:x\in\mathcal{X}\mapsto \argmin_{y\in\mathcal{X}} \left\{\phi(y)+\frac{1}{2}\sqn{y-x}\right\}$,  which is well defined and single-valued \citep{par14}. We study fully split primal--dual proximal algorithms; that is, methods that solve the primal and dual problems \eqref{eq1} and \eqref{eq1cda} simultaneously by iteratively evaluating
$\K$, $\K^*$, $\nabla \ff$, and the proximity operators of $\g$ and $\h$ (or equivalently $\hc$ via the Moreau identity 
$x=\mathrm{prox}_{\h/\tau} (x)+\frac{1}{\tau} \mathrm{prox}_{\tau \hc}(\tau x)$
\citep[Theorem 14.3(ii)]{bau17}). 
In particular, no linear system or inner subproblem involving $\K$ is ever solved.

For the primal and dual problems \eqref{eq1}--\eqref{eq1cda} to be well posed, 
we assume throughout that there exists $(x^\star,u^\star)\in\mathcal{X}\times\mathcal{U}$ satisfying the optimality conditions
\begin{equation}
\left\{\begin{array}{l}
0 \in  \nabla \ff(x^\star)+{\partial \g(x^\star)}+\K^*u^\star,\\
0 \in \partial \hc(u^\star)-\K x^\star.\end{array}\right.\label{eq2}
\end{equation}
Under these conditions, $x^\star$ is a solution of \eqref{eq1} and $u^\star$ is a solution of \eqref{eq1cda}. 
Moreover, if a suitable constraint qualification holds, 
then \eqref{eq2} is not only sufficient but also necessary for $(x^\star,u^\star)$ to be a primal--dual solution \citep[Section 27.1]{bau17}.

This work investigates how Nesterov-type acceleration \citep{nes13} can be effectively incorporated into primal--dual methods for this class of problems.
Our guiding viewpoint is the following. Nesterov acceleration, in its proximal form for minimizing $\ff+\g$, evaluates the gradient $\nabla \ff$ and, as we see it, exploits the strong convexity of the dual problem $\min_u 
\ffc(-u)+\gc(u)$: since $\ff$ is $L_\ff$-smooth, $\ffc$ is $L_\ff^{-1}$-strongly convex, 
so that this dual problem is strongly convex, whatever $\g$ is. 
The dual problem \eqref{eq1cda} of Problem~\eqref{eq1} is strongly convex in several situations: when $\h$ is smooth; when $\K^*$ is bounded below and $\g$ is smooth; and, for linearly constrained problems with $\g$ smooth, after restriction of the dual variable to the range of $\K$. 
Our goal is to extend Nesterov acceleration to these situations, while retaining the fully split structure described above.

Section~\ref{secne} reviews Nesterov's acceleration of proximal gradient descent 
for the minimization of $\ff+\g$. 
Section~\ref{sec2} then introduces our new algorithm, which achieves the same acceleration with respect to $\ff$ while accommodating the composite term $\h\circ\K$.
 In the remainder of this section, we review the state of the art and detail our contributions.

\subsection{Related Work}\label{secrel}

Algorithms that solve nonsmooth optimization problems involving several functions through iterative evaluations of gradients and proximity operators are known as 
\emph{proximal splitting algorithms} \citep{com10,bot14,par14,%kom15,
bec17,com21}.
For Problem~\eqref{eq1}, which involves the composite term $\h\circ\K$, several primal--dual algorithms have been proposed. 
We refer to the surveys \citet{con23,com24}  and to the recent developments in  \citet{sal20,con22}. 
We briefly review the most relevant methods, focusing on their convergence rates. 

\paragraph{Proximal Gradient Descent.}
A prominent method for the simple problem of minimizing $\ff+\g$ is  \emph{Proximal Gradient Descent} (\algn{PGD}), which iterates
$x^{t+1}  \coloneqq \mathrm{prox}_{\gamma\g}\left(x^t-\gamma\nabla \ff(x^t)\right)$, 
where $t\geq 0$ is the iteration index and $\gamma >0$ the stepsize. It is known that if $\gamma \in \big(0,\frac{2}{L_\ff}\big)$, \algn{PGD} achieves a convergence rate  $\mathcal{O}(1/t)$ \citep{bub15,tay18}.  
If, in addition, $\mu_\ff>0$ or $\mu_\g>0$, \algn{PGD} converges linearly to the unique solution $x^\star$, with
$\|x^{t+1}-x^\star\|\leq \frac{\max(1-\gamma\mu_\ff,\gamma L_\ff-1)}{1+\gamma\mu_\g}\|x^t-x^\star\|$ \citep[Theorem 3.1]{tay18}, \citep[Proposition 23.13]{bau17}.
These rates are not accelerated and can be improved using Nesterov-type acceleration \citep{nes13}, as discussed in Section~\ref{secne}.

\paragraph{Primal--Dual Hybrid Gradient and Condat--V\~u algorithms.}
A key method is the \emph{Primal--Dual Hybrid Gradient} algorithm (\algn{PDHG})  for minimizing 
$\g+\h\circ\K$.  A modified version of  \algn{PDHG} that requires knowledge of $\mu_\g$ converges linearly when $\mu_\g>0$ and $\h$ is $L_\h$-smooth. When the stepsizes are chosen optimally, its complexity, as a number of iterations to reach a given accuracy,
 is accelerated and scales as $\sqrt{\|\K\|^2 L_\h/\mu_\g }$ \citep{cha11a}.%
\footnote{For a composite problem involving an $L$-smooth function and a $\mu$-strongly convex function, we call $L/\mu$ its condition number; this is the quantity that governs the complexity of accelerated first-order methods, see Section~\ref{secne}, where it is $L_\ff/\mu_\g$ for the minimization of $\ff+\g$. Here, $\h\circ \K$ is $\|\K\|^2 L_\h$-smooth and $\g$ is $\mu_\g$-strongly convex, so that the condition number of the primal problem is $\|\K\|^2 L_\h/\mu_\g$. The dual problem has the same condition number, since $\hc$ is $1/L_\h$-strongly convex and $\gc \circ (-\K^*)$ is $\|\K\|^2/\mu_\g$-smooth. Thus, \algn{PDHG} can be regarded as accelerated in the sense that its complexity depends on the square root of this primal--dual condition number.}
The unmodified \algn{PDHG} has been shown to converge linearly under the same assumptions in \citet{oco20} and \citet[Theorem 7]{con22rp}.

To extend  \algn{PDHG} to Problem~\eqref{eq1} with $\ff\neq 0$,  the  \emph{Condat--V\~u} algorithm (\algn{CV}) was proposed independently by Condat \citep{con13} and V\~u \citep{vu13}. \algn{CV} reduces to \algn{PDHG} when $\ff=0$. With stepsizes $\gamma>0$ and $\tau>0$, it iterates
\begin{align*}
x^{t+1}&\coloneqq \mathrm{prox}_{\gamma\g}\big(x^t-\gamma\nabla \ff(x^t)-\gamma\K^* u^{t}\big), \\
u^{t+1}&\coloneqq \mathrm{prox}_{\tau \hc} \big(u^t+\tau \K (2x^{t+1}-x^t)\big) 
\end{align*}
(up to a symmetric form obtained by swapping the order of the updates).
 Linear convergence of 
\algn{CV} has been studied in \citet{cha162,dir25}.

\paragraph{Proximal Alternating Predictor--Corrector and three-operator splitting algorithms.}
Another method for Problem~\eqref{eq1} with $\g=0$ is the \emph{Proximal Alternating Predictor--Corrector} algorithm (\algn{PAPC}) \citep{dro15}, also proposed independently as 
\algn{PDFP${^2}$O} \citep{che13} and in \citet{lor11}  for regularized least-squares problems. 
With stepsizes $\gamma>0$ and $\tau>0$, \algn{PAPC}  iterates
\begin{align}
\hat{x}^t&\coloneqq x^t-\gamma\nabla \ff(x^t)-\gamma\K^* u^t, \notag\\
u^{t+1}&\coloneqq \mathrm{prox}_{\tau \hc} \big(u^t+\tau \K \hat{x}^t\big), \label{eqpapc}\\ 
x^{t+1}&\coloneqq x^t-\gamma\nabla \ff(x^t)-\gamma\K^* u^{t+1}. \notag
\end{align}
The method converges whenever $\gamma \in \big(0,\frac{2}{L_\ff}\big)$ and $\gamma\tau\|\K\|^2\leq 1$, which provides a  larger admissible stepsize range than \algn{CV} \citep{con23}. Both \algn{CV} and \algn{PAPC} can be interpreted as preconditioned forward--backward iterations in the product space $\mathcal{X}\times\mathcal{U}$ 
\citep{com14}, differing only in the choice of preconditioner. 

When $\K=\Id$ is the identity,  the \emph{Three-Operator Splitting} algorithm (\algn{TOS}) is well suited to minimize $\ff+\g+\h$ \citep{dav17}, and its linear convergence has been studied in \citet{lee25}. More generally, two ``twin'' algorithms have been proposed for Problem~\eqref{eq1}, recovering \algn{PDHG}, \algn{PAPC} and \algn{TOS} as special cases: 
the \emph{Primal--Dual Three-Operator Splitting} algorithm (\algn{PD3O})  \citep{yan18}
and the \emph{Primal--Dual Davis--Yin} algorithm (\algn{PDDY}) \citep{sal20}.
Both methods converge under the same stepsize conditions as \algn{PAPC}. 
Their linear convergence was established in \citet{con22}, and
 refined rates for \algn{PDDY} (and hence \algn{PDHG} and \algn{PAPC}) were derived in \citet{con22rp}.

When $\mu_\ff>0$ or $\mu_\g>0$ and $\h$ is $L_\h$-smooth, the complexity of  \algn{CV},  \algn{PD3O} and \algn{PDDY}, optimized over the stepsizes, scales as
$\frac{L_\ff}{\mu_\ff+\mu_\g}+\sqrt{\frac{\|\K\|^2 L_\h}{\mu_\ff+\mu_\g }}$. The second term exhibits accelerated dependence on the primal--dual condition number, whereas the first term matches the complexity of \algn{PGD}. 
This behavior is expected, since these algorithms treat $\ff$ via a standard gradient descent step $x^t-\gamma\nabla\ff(x^t)$.

\paragraph{Acceleration of primal--dual algorithms.}
A natural question is therefore whether these primal--dual schemes can be accelerated so as to achieve improved dependence on $L_\ff$. Let us delineate the scope of this question. We focus on the case $\ff\neq 0$ and on acceleration with respect to $\nabla \ff$; when $\ff=0$, the question becomes that of the inertial acceleration of \algn{PDHG}, which is a different topic. We also focus on fully split algorithms; several accelerated methods have been proposed for the linearly constrained case of Problem~\eqref{eq1}, 
which require solving a linear system involving $\K$ at every iteration \citep{bot23,he24,luo26}, and we do not discuss them further. 

Acceleration of primal--dual algorithms is notoriously challenging: they exhibit rotational dynamics in the primal--dual space, which naive momentum exacerbates, causing divergence, as confirmed by the spectral analysis in \citet{luo24}.
 Accelerated variants of \algn{CV} have been proposed, initially for $\g=0$ \citep{che142} and later in the general case \citep{zha19,zha22}, but without linear convergence guarantees. These methods, like the \emph{Accelerated Linearized ADMM} \citep{ouy15} or \emph{Accelerated Gradient Sliding} \citep{lan22}, do not assume strong convexity of the dual problem, and their acceleration is only partial: their complexity scales as $\sqrt{L_\ff/\epsilon}+\|\K\|D/\epsilon$ for some problem-dependent constant $D$, so that the dependence on $L_\ff$ is accelerated, but not the dependence on $\|\K\|$. 
This is unavoidable in general: even for linearly constrained problems, the term $\|\K\|D/\epsilon$ cannot be improved by methods that only call $\nabla\ff$, $\K$, and $\K^*$, as long as the constants are expressed in terms of $\|\K\|$ only \citep{ouy21}. In other words, the dual of a linearly constrained problem with smooth objective is strongly convex on the range of $\K$, with a modulus governed by the smallest nonzero eigenvalue $\lambda^+_{\min}(\K\K^*)$ of $\K\K^*$, but a guarantee that holds uniformly with respect to this modulus, as if it could be arbitrarily close to zero, is necessarily sublinear in the absence of strong convexity. A recent example is the \emph{Fast Reflected Forward--Backward} algorithm \citep{bot253}, a fully split method for Problem~\eqref{eq1} in full generality.
It achieves weak convergence of the iterates together with non-ergodic $o(1/t)$ rates for the Lagrangian gap, 
as well as for the feasibility gap and the objective values of linearly constrained problems, with constants depending on $L_\ff$ and $\|\K\|$ only. 
In this work, by contrast, we consider the situations where the dual problem is strongly convex
 and exploit the knowledge of its strong convexity modulus to aim at a complete acceleration, with $\mathcal{O}(1/t^2)$ rates, or linear rates when the primal problem is strongly convex too. For linearly constrained problems, this means that $\lambda^+_{\min}(\K\K^*)$, or a lower bound on it, is used as a parameter of the algorithm.

Fully split algorithms achieving a complete acceleration have been proposed in the regime where $\h$ is smooth. Recently, a comprehensive analysis of the \emph{Accelerated Condat--V\~u} algorithm (\algn{ACV}) was presented \citep{dri24}. Notably, it achieves accelerated linear convergence when $\mu_\g>0$ and $\h$ is $L_\h$-smooth, with complexity scaling as
 $\sqrt{L_\ff/\mu_\g}+\sqrt{\|\K\|^2 L_\h/\mu_\g }$. In the same regime, two other fully split, single-loop algorithms with Nesterov-type momentum achieve $\mathcal{O}(1/t^2)$ convergence of the Lagrangian gap, even with a general nonsmooth $\g$: the inertial primal--dual algorithm of \citet{he25}, with \algn{FISTA}-type extrapolation of both the primal and dual variables, and the accelerated Bregman primal--dual proximal gradient splitting algorithm of \citet{luo24}, derived from a primal--dual dynamical system, which also achieves linear convergence when $\ff$ is strongly convex, with the same complexity as \algn{ACV}.
 In the linearly constrained regime, an algorithm termed \algn{APAPC} was introduced by \citet{kov20} for the specialized problem $\min \ff + \frac{\mu_\g}{2}\|\cdot\|^2$ subject to $\K^* \K x=0$ (with $\mu_\g>0$), which arises in decentralized optimization. This method was later extended to general linear constraints \citep{sal22}, with accelerated linear convergence; we revisit this setting in Section \ref{seclc}. To the best of our knowledge, no fully split algorithm with an $\mathcal{O}(1/t^2)$ rate is available for linearly constrained problems without strong convexity, nor in the regime where $\K^*$ is bounded below.
 
We also acknowledge the concurrent work of \citet{zhu25}, who proposed the \emph{Accelerated Primal--Dual Fixed-Point} method (\algn{APDFP}). When $\mu_\g=0$, their algorithm coincides with our \algn{APAPC}. Their analysis, however, belongs to the partially accelerated family discussed above: it assumes that $\h$ is Lipschitz continuous, so that $\hc$ has a bounded domain, 
to derive an $\mathcal{O}(L_\ff/t^2+\|\K\|D/t)$ bound on the partial primal--dual gap. We do not impose such assumptions and exploit the strong convexity of the dual problem instead.

 \subsection{Contributions}
 
 In this work, we propose the \emph{Accelerated Proximal Alternating Predictor--Corrector} algorithm (\algn{APAPC}), an accelerated version of \algn{PAPC}, 
to solve Problem \eqref{eq1} under the specific assumption that $\g = \frac{\mu_\g}{2}\|\cdot\|^2$ for some $\mu_\g \geq 0$. This restriction is dictated by our viewpoint on acceleration as a technique exploiting the strong convexity of the dual problem. Indeed, $(\ff+\g)^*$ is strongly convex if and only if $\ff+\g$ is smooth \citep[Theorem 18.15]{bau17}, so that with a nonsmooth $\g$, the dual problem \eqref{eq1cda} can only be strongly convex through $\hc$, that is, when $\h$ is smooth. With a quadratic $\g$, $\ff+\g$ remains smooth, and the two other regimes of dual strong convexity, namely $\K^*$ bounded below and linearly constrained problems, become available. We note that in these two regimes, linear convergence cannot be achieved with an arbitrary nonsmooth $\g$, even if it is strongly convex \citep[Section VI]{alg21}. Thus, our structured setting captures the key difficulties while covering the three regimes. 
The regime where $\h$ is smooth, the only one compatible with a general nonsmooth $\g$, is also the only regime in which the existing accelerated algorithms for Problem~\eqref{eq1}, such as \algn{ACV}, have been analyzed \citep{dri24,luo24,he25}.
 Handling a nonsmooth $\g$ with a predictor--corrector scheme would require a more complex splitting such as \algn{PD3O} or \algn{PDDY}; we do not pursue this direction here and focus instead on a unified treatment of the three regimes. In the linearly constrained regime, \algn{APAPC} is similar to the earlier method of the same name in \citet{kov20,sal22}, where only the case $\mu_\g>0$ is studied; our \algn{APAPC} not only covers the case $\mu_\g=0$ but also extends to an arbitrary nonsmooth composite term $\h \circ \K$.

Our main contributions are structured as follows:
\begin{itemize}
\item \textbf{A Unified Analysis of Nesterov Acceleration:} In Section~\ref{secne}, we revisit Nesterov acceleration for the minimization of $\ff+\g$. 
Our refined convergence analysis of Accelerated Proximal Gradient Descent (\algn{APGD}) offers a unified treatment of both the general convex and strongly convex regimes, achieving optimal linear convergence simply by capping the momentum parameter. Moreover, we prove weak convergence of the iterates to a solution. This framework lays the foundation for our primal--dual extension.
\item \textbf{Seamless Primal--Dual Acceleration: } We propose \algn{APAPC} for solving Problem \eqref{eq1} with $\g = \frac{\mu_\g}{2}\|\cdot\|^2$, by integrating the decoupled momentum architecture of \algn{APGD} into the forward--backward structure of \algn{PAPC}. Exploiting the strong convexity of the dual problem allows for acceleration on the smooth component $\ff$ while ensuring the decay of the primal--dual Lyapunov function.
\item \textbf{Comprehensive Convergence Rates:} We establish accelerated convergence bounds across three distinct regimes where the dual problem is strongly convex. We prove accelerated linear convergence when the primal problem is also strongly convex ($\mu_\g>0$), and optimal $\mathcal{O}(1/t^2)$ sublinear rates otherwise. 
While we recover known complexity bounds in specific cases, such as when $\h$ is smooth \citep{dri24,luo24} or for linearly constrained strongly convex problems \citep{sal22}, our $\mathcal{O}(1/t^2)$ rates when $\mu_\g=0$ and $\K^*$ is bounded below or the problem is linearly constrained are, to the best of our knowledge, new for fully split algorithms.
Our results are summarized in the following table:

{
\renewcommand{\tabcolsep}{0.15cm}
\begin{tabular}{l||l|l|l}
\hline
Condition on $\h\circ\K$ &$\h$ smooth & $\K^*$ bounded below & Linear constraint\\ 
\hline
$\mathcal{O}(1/t^2)$ convergence& Theorem \ref{corcase1a}& Theorem \ref{theoinj}& Theorem \ref{theolc}\\
\hline
Linear convergence when $\mu_\g>0$&Corollary \ref{corcase1b}&Corollary \ref{corcase2b}&Corollary \ref{corcase3b}\\
\hline
\end{tabular}}
\item \textbf{Point Convergence: }Building upon our Lyapunov framework and recent results on Accelerated Gradient Descent, we characterize the weak convergence of the iterates for \algn{APAPC} in the three regimes (Theorem~\ref{theolcp} and Remark~\ref{remconvother}).
 \end{itemize}

\section{Revisiting Nesterov Acceleration}\label{secne}

In this section, we consider the simple particular case of \eqref{eq1} with $\h \circ \K=0$:
\begin{equation}
\minimize_{x\in\mathcal{X}} \;\Psi(x)\coloneqq \ff(x)+\g(x). \label{eq1p}
\end{equation}
In that case, $x^\star\in\mathcal{X}$
 is a minimizer of $\Psi$ if and only if it satisfies $0\in {\nabla \ff(x^\star)+{\partial \g(x^\star)}}$ \citep[Corollary 16.48]{bau17}. 
We suppose that a minimizer $x^\star$ of $\Psi$ exists and let $\Psi^\star\coloneqq \Psi(x^\star)$.

We consider the \emph{Accelerated Proximal Gradient Descent} algorithm (\algn{APGD}). Given a stepsize $\gamma>0$ and a parameter sequence $(a_t)_{t\geq 0}$ with $a_t\geq 1$ for every $t\geq 1$,   
\algn{APGD}, initialized with $z^0=x^0$, iterates, for $t\geq 0$,
\begin{align}
y^{t} & \coloneqq \left(1-\frac{1}{a_{t+1}}\right) x^{t}+\frac{1}{a_{t+1}} z^{t},\label{eqala1}\\
z^{t+1} & \coloneqq \mathrm{prox}_{a_{t+1}\gamma\g}\left(z^t-a_{t+1}\gamma\nabla \ff(y^t)\right),\label{eqala1p}\\
x^{t+1} & \coloneqq \left(1-\frac{1}{a_{t+1}}\right) x^{t}+\frac{1}{a_{t+1}} z^{t+1}. 
\label{eqala1p3}
\end{align}
Note that \eqref{eqala1p3} can be rewritten as
\begin{equation}
z^{t+1}=x^{t+1}+(a_{t+1}-1)(x^{t+1}-x^t),\quad\forall t\geq 0,\label{eqzx}
\end{equation}
a form that will be used repeatedly.

\algn{APGD} is a proximal extension of Nesterov's Accelerated Gradient Descent (\algn{AGD}) \citep[p.~90]{nes04}, which appeared in \citet[Section 5]{aus06} and was studied in \citet{tse08}.
In \algn{APGD}, $\gamma$ is typically chosen as $\frac{1}{L_\ff}$, but $a_{t+1}\geq 1$ can be arbitrarily large. 
Thus, $z^t$ is updated using a proximal gradient descent step with an aggressive stepsize, which is then damped by under-relaxation, as $x^{t+1}$ is a convex combination of this new point and the previous conservative estimate $x^t$ \citep{all14}. Thus, $x^t$ is a weighted average of all estimates $z^0,\ldots, z^t$ computed so far. For instance, we can check by induction that if $a_0=0$ and $a_t=\frac{t+1}{2}$, then $x^t = \sum_{k=1}^t \frac{2k}{t(t+1)} z^k$ for every $t\geq 1$.
If $x^1$ is feasible, i.e., in the domain of $\g$, then $x^t$ remains feasible for every $t\geq 1$, which is a desirable property. 
This is the case if $x^0$ is feasible or if $a_1=1$, as we recommend in practice, since then $x^1=z^1$.

We observe that if $a_{t+1}\equiv 1$, then $x^t=z^t=y^t$ and \algn{APGD} reverts to \algn{PGD}.

We note that \algn{APGD} is different from the popular method \algn{FISTA} \citep{bec092}. After constructing $y^t$, \algn{FISTA} updates $z^t$ and $x^t$ as
\begin{align*}
x^{t+1} & \coloneqq \mathrm{prox}_{\gamma\g}\left(y^t-\gamma\nabla \ff(y^t)\right), \\
z^{t+1} & \coloneqq z^t + a_{t+1} (x^{t+1}-y^t).  
\end{align*}
\algn{APGD} and \algn{FISTA} are different in the way $\mathrm{prox}_{\g}$ is applied. In both methods, the variable of interest to estimate the solution is $x^t$, not $z^t$, see Remark~\ref{rem15}.  
It may be that $x^t$ lies on the boundary of the domain of $\g$ in  \algn{FISTA}, while it is in the interior of this domain in \algn{APGD}. For instance, if $\g$ induces sparsity, $x^t$ is typically more sparse in  \algn{FISTA} than in \algn{APGD}. Explicitly promoting structure in the estimate 
is a property that certainly contributed to \algn{FISTA}'s widespread adoption in inverse problems.
However, after $T$ iterations of \algn{APGD}, one can always apply one final iteration of \algn{PGD} to $x^T$. Since \algn{PGD} is a descent method, we have $\Psi(x^{T+1})\leq \Psi(x^T)$, and the final estimate $x^{T+1}$, as the output of  $\mathrm{prox}_{\gamma\g}$, now has the sought properties, 
such as being sparse or low rank. The structural properties of the last iterate matter, and it may not be necessary to impose them at every intermediate iteration. In any case, \algn{APGD}'s decoupling of the estimate $x^t$, which undergoes momentum, from the history-accumulating sequence $z^t$,  which can accommodate 
dual variables, is a key property of our primal--dual extension in Section \ref{sec2}.

If $\g=0$, \algn{APGD} and \algn{FISTA} coincide and revert to \algn{AGD}:
\begin{align*}
y^{t} & \coloneqq x^{t}+\frac{a_t-1}{a_{t+1}}(x^{t}-x^{t-1}),\\
x^{t+1} & \coloneqq y^t-\gamma \nabla \ff(y^t),
\end{align*}
with $z^t$ given by \eqref{eqzx}. 
We refer to \citet{bub15} and \citet[Appendix B]{cha16} for the convergence analysis of \algn{AGD} and \algn{FISTA}.

While the convergence properties of \algn{APGD} have been studied \citep{tse08}, 
we now provide a simple convergence analysis based on a single-iteration Lyapunov inequality.
 To the best of our knowledge, this analysis, which captures both the general convex and strongly convex regimes while decoupling $\mu_\ff$ and $\mu_\g$, is new. Let us emphasize the benefits of such a unified treatment. First, a single algorithm with a single parameter rule, namely $a_t=\min\big(\frac{t+1}{2},a_\sharp\big)$ for some cap $a_\sharp$, which is finite only when $\mu_\g>0$, covers both regimes; when $\mu_\g>0$, this rule yields a first phase of sublinear decay, faster than the asymptotic linear rate, before the latter takes over, see Remark~\ref{rempract}. Second, the strongly convex regime comes at no extra cost: the same Lyapunov inequality yields linear convergence simply by capping the growth of $a_t$, and the rate transfers to $\sqn{x^t-x^\star}$ and $\Psi(x^t)-\Psi^\star$, so that strong convergence of the iterates with an explicit rate is obtained in this case; weak convergence, studied in Section~\ref{secpc}, is only needed when $\mu_\g=0$. Third, and most importantly for our purposes, this structure, a single-iteration inequality from which every rate follows by a choice of the sequence $(a_t)_{t\geq 0}$, is exactly what carries over to the primal--dual setting of Section~\ref{sec2}, where the single inequality of Theorem~\ref{theo2} yields the rates in six regimes, namely three types of dual strong convexity, with or without $\mu_\g>0$.
  
\begin{theorem}[Single-iteration progress of \algn{APGD}]\label{theo1}
Let $x^\star\in\argmin \Psi$. In \algn{APGD}, we have
\begin{align}
\frac{1+a_{t+1}\gamma\mu_\g}{2\gamma}\|z^{t+1}-x^\star\|^2+a_{t+1}^2 \big(\Psi(x^{t+1})-\Psi^\star\big)&\leq\frac{1-\gamma\mu_\ff}{2\gamma}\|z^{t}-x^\star\|^2+(a_{t+1}^2-a_{t+1})\big(\Psi(x^{t})-\Psi^\star)\notag\\
&\quad-{\left(\frac{1}{2\gamma}-\frac{L_\ff}{2}\right)}\|z^{t+1}-z^t\|^2,\quad \forall t\geq 0.\label{eqin1}
\end{align}
Consequently, we define the Lyapunov function
\begin{equation*}
\mathcal{E}^t_{x^\star}\coloneqq \frac{1+a_{t}\gamma\mu_\g}{2\gamma}\|z^{t}-x^\star\|^2+a_{t}^2 \big(\Psi(x^{t})-\Psi^\star\big),\quad\forall t\geq 0.
\end{equation*}
Then, if $\gamma\leq \frac{1}{L_\ff}$, the last term in \eqref{eqin1} can be discarded, and we have
\begin{equation*}
\mathcal{E}^{t+1}_{x^\star} \leq \max\left(\frac{1-\gamma\mu_\ff}{1+a_{t}\gamma\mu_\g},\frac{a_{t+1}^2-a_{t+1}}{a_t^2}\right)\mathcal{E}^t_{x^\star} ,\quad\forall t\geq 1,
\end{equation*}
and also for $t=0$ if $a_0>0$. If $a_0=0$ and $a_1=1$, \eqref{eqin1} with $t=0$ directly yields $\mathcal{E}^{1}_{x^\star}\leq \frac{1-\gamma\mu_\ff}{2\gamma}\|z^{0}-x^\star\|^2\leq \mathcal{E}^{0}_{x^\star}$. Here and in the sequel, we use the convention $0\cdot(+\infty)=0$, so that $\mathcal{E}^{0}_{x^\star}$ and \eqref{eqin1} with $t=0$ are well defined even if $a_0=0$ and $x^0\notin \mathrm{dom}\,\g$.
\end{theorem}

The proof of Theorem \ref{theo1} 
is simple once we have the following lemma, proved in Section~\ref{secprooflem1}.
 In \algn{APGD}, we introduce, for every $t\geq 0$, the subgradient $\tilde{\nabla}\g(z^{t+1})\in \partial \g(z^{t+1})$ such that $z^{t+1}+a_{t+1}\gamma \tilde{\nabla}\g(z^{t+1})
=z^t-a_{t+1}\gamma\nabla \ff(y^t)$; it is well defined by construction of $z^{t+1}$ in \eqref{eqala1p} \citep[Proposition 16.44]{bau17}. Given $x^\star\in\argmin \Psi$, we also introduce, for every $t\geq 0$, the nonnegative quantities
\begin{equation*}
\sigma_t\coloneqq \frac{\sqn{\nabla \ff(x^t) - \nabla \ff(y^t) - \mu_\ff(x^t - y^t)}}{2(L_\ff-\mu_\ff)}\quad\mbox{and}\quad \sigma^\star_t\coloneqq \frac{\sqn{\nabla \ff(y^t) - \nabla \ff(x^\star) - \mu_\ff(y^t - x^\star)}}{2(L_\ff-\mu_\ff)},
\end{equation*}
with the convention $\sigma_t=\sigma^\star_t=0$ if $\mu_\ff=L_\ff$. These terms are discarded in the convergence results of this section; they will be used in Section~\ref{sec2}.

\begin{lemma}\label{lemman1}
In \algn{APGD}, 
for every $x^\star\in\argmin \Psi$ and $t\geq 0$, we have
\begin{align*}
-a_{t+1}\big\langle \tilde{\nabla}\g(z^{t+1})+\nabla \ff(y^t),z^{t+1}-x^\star\big\rangle
&\leq-a_{t+1}^2 \big(\Psi(x^{t+1})-\Psi^\star\big)+(a_{t+1}^2-a_{t+1})\big(\Psi(x^{t})-\Psi^\star\big)\\
&\quad-\textstyle{\frac{a_{t+1}\mu_\g}{2}\|z^{t+1}-x^\star\|^2-{\frac{\mu_\ff}{2}}\|z^t -x^\star\|^2+\frac{L_\ff}{2}
\|z^{t+1}-z^t\|^2}\\
&\quad-a_{t+1}\sigma^\star_t-(a_{t+1}^2-a_{t+1})\sigma_t.
\end{align*}
\end{lemma}

\begin{proof}[Proof of Theorem \ref{theo1}]
Let $x^\star\in\argmin \Psi$ 
and $t\geq 0$. 
We have $z^{t+1}-z^t=-a_{t+1}\gamma\big(\tilde{\nabla}\g(z^{t+1})
+\nabla \ff(y^t)
\big)$. Therefore,
\begin{align*}
\|z^{t+1}-x^\star\|^2&=\|z^{t}-x^\star\|^2-\|z^{t+1}-z^t\|^2+2\left\langle z^{t+1}-z^t,z^{t+1}-x^\star\right\rangle\\
&=\|z^{t}-x^\star\|^2-\|z^{t+1}-z^t\|^2-2a_{t+1}\gamma \big\langle \tilde{\nabla}\g(z^{t+1})+\nabla \ff(y^t),z^{t+1}-x^\star\big\rangle.
\end{align*}
Dividing by $2\gamma$ on both sides and invoking Lemma~\ref{lemman1}, we obtain the inequality \eqref{eqin1}.
\end{proof}

We can now focus on the choice of  the sequence $(a_t)_{t\geq 0}$.
\begin{corollary}\label{cor13}
Under the conditions of Theorem~\ref{theo1}, suppose that $\gamma\leq \frac{1}{L_\ff}$, $a_0=0$,
and $a_{t+1}^2-a_{t+1}\leq a_t^2$ (equivalently,
$a_{t+1} \leq \frac{1+\sqrt{1+4a_t^2}}{2}$)
for every $t\geq 0$ (this implies $a_1=1$). Then
\begin{equation*}
\mathcal{E}^{t}_{x^\star} \leq \mathcal{E}^0_{x^\star}=\frac{1}{2\gamma}\|x^0-x^\star\|^2 ,\quad\forall t\geq 0,
\end{equation*}
so that
\begin{equation*}
\Psi(x^t)-\Psi^\star \leq \frac{\sqn{x^0-x^\star}}{2\gamma a_t^2} ,\quad\forall t\geq 1.
\end{equation*}
In particular, by choosing $\gamma=\frac{1}{L_\ff}$ and $a_t=\frac{t+1}{2}$ (or the slightly larger sequence defined recursively by $a_{t+1}\coloneqq \frac{1+\sqrt{1+4a_t^2}}{2}$) for every $t\geq 1$, which satisfies the condition above, \algn{APGD} has accelerated sublinear convergence with
\begin{equation}
\Psi(x^t) -\Psi^\star\leq 
\frac{2 L_\ff \sqn{x^0-x^\star}}{(t+1)^2}.\label{eqratf}
\end{equation}
Moreover, suppose that $\mu_\g>0$, choose $\gamma=\frac{1}{L_\ff}$, and let $T\geq 1$ be such that $a_t= a_\sharp \coloneqq \max\left(\sqrt{\frac{L_\ff}{\mu_\g}},1\right)$ for every $t\geq T$.
 Then $x^\star=\argmin \Psi$ is unique and 
\algn{APGD} has accelerated linear convergence with
\begin{equation*}
\mathcal{E}^{t}_{x^\star} \leq \left(\frac{1}{1+\sqrt{\mu_\g/L_\ff}}\right)^{t-T}\mathcal{E}^T_{x^\star} ,\quad\forall t\geq T,
\end{equation*}
hence an iteration complexity (to reach an accuracy $\mathcal{E}^t_{x^\star}\leq \epsilon$) $\mathcal{O}\left(\sqrt{\frac{L_\ff}{\mu_\g}}\log \epsilon^{-1} \right)$.
Moreover, since $\Psi(x^t)-\Psi^\star\geq \frac{\mu_\g}{2}\sqn{x^t-x^\star}$ by strong convexity of $\Psi$, the quantities $\sqn{z^t-x^\star}$, $\sqn{x^t-x^\star}$, and $\Psi(x^t)-\Psi^\star$ all converge linearly, with the same rate.
\end{corollary}

These asymptotic rates are known to be optimal for first-order methods \citep{nes04}. 
The constant 2 in \eqref{eqratf} can be removed with the refined and optimal method \algn{OptISTA} \citep{jan25}.
We also note, for later use, that the condition $a_{t+1}^2-a_{t+1}\leq a_t^2$ implies
\begin{equation}
a_{t+1}\leq \frac{1+\sqrt{1+4a_t^2}}{2}\leq \frac{1+(1+2a_t)}{2}= a_t+1,\quad\forall t\geq 0,\label{eqagrow}
\end{equation}
so that $a_t\leq t$ for every $t\geq 0$ if $a_0=0$. 

\begin{remark}\label{rempract}
In practice, we can choose
 $\gamma=\frac{1}{L_\ff}$, $a_0=0$ and  $
a_{t}= \min\left( \frac{t+1}{2}, a_\sharp \right)$ (or $a_{t}=\min\Big( \frac{1+\sqrt{1+4a_{t-1}^2}}{2}, a_\sharp \Big))$
for every $t\geq 1$. This way, if $\mu_\g>0$, 
there is a first phase of about $2a_\sharp$ iterations during which the sublinear decrease of $\Psi(x^t)$ is faster than the linear rate, before the latter takes over.
\end{remark}
\begin{remark}\label{rem15}
Under the conditions of Corollary~\ref{cor13} with $\mu_\g=0$,
 we have $\Psi(x^t)\rightarrow \Psi^\star$, but there is no guarantee that $\sqn{z^t-x^\star}$ decreases, in contrast to the strongly convex case. Thus, $x^t$ is indeed the variable of interest, not $z^t$.
\end{remark}
\begin{remark}\label{remmuf}
We observe in Theorem~\ref{theo1} that if $\mu_\g=0$ and $\mu_\ff>0$, \algn{APGD} converges linearly with rate at best $1-\gamma\mu_\ff$, which is not accelerated. 
However, accelerated convergence can be obtained in this case by transferring the strong convexity from $\ff$ to $\g$
 by writing $\ff+\g = \tilde{\ff}+\tilde{\g}$, with $\tilde{\ff}=\ff-\frac{\mu_\ff}{2}\|\cdot\|^2$ and $\tilde{\g}=\g+\frac{\mu_\ff}{2}\|\cdot\|^2$, and doing the corresponding changes in the algorithm.
In Section~\ref{sec2}, where $\g$ is quadratic, this transfer costs nothing.
\end{remark}
\begin{remark}
When $\mu_\g>0$, the linear convergence rate can be improved by a factor $\sqrt{2}$ by slightly modifying the analysis.
 For this, given $x^\star$, we define $\tilde{\g}=\g-\frac{\mu_\g}{2}\|\cdot-x^\star\|^2$. Then, in the proof of Theorem~\ref{theo1}, we have
\begin{align*}
\|z^{t+1}-x^\star\|^2&=\|z^{t}-x^\star\|^2-\|z^{t+1}-z^t\|^2-2a_{t+1}\gamma \mu_\g \sqn{z^{t+1}-x^\star}\\
&\quad -2a_{t+1}\gamma \big\langle \tilde{\nabla}\tilde{\g}(z^{t+1})+\nabla \ff(y^t)
,z^{t+1}-x^\star\big\rangle.
\end{align*}
Then, invoking Lemma~\ref{lemman1} with $\g$ replaced by $\tilde{\g}$, we obtain, if $\gamma\leq \frac{1}{L_\ff}$,
\begin{align*}
\frac{1\!+\!2a_{t+1}\gamma\mu_\g}{2\gamma}\|z^{t+1}-x^\star\|^2+a_{t+1}^2  \big(\widetilde{\Psi}(x^{t+1})-\Psi^\star\big) 
&\leq\frac{1\!-\!\gamma\mu_\ff}{2\gamma}\|z^{t}-x^\star\|^2+(a_{t+1}^2\!-\!a_{t+1}) \big(\widetilde{\Psi}(x^{t})-\Psi^\star\big), 
\end{align*}
where $\widetilde{\Psi}=\ff+\tilde{\g}=\Psi -\frac{\mu_\g}{2}\|\cdot-x^\star\|^2$. Note the factor 2 in  $1+2a_{t+1}\gamma\mu_\g$, in comparison with  $1+a_{t+1}\gamma\mu_\g$ in \eqref{eqin1}. 
Therefore, in Corollary \ref{cor13}, we can choose $a_\sharp \coloneqq \max\Big(1,\sqrt{L_\ff/(2\mu_\g)}\Big)$ and we get the contraction factor $\rho\coloneqq\Big(1+\sqrt{2\mu_\g/L_\ff}\Big)^{-1}$ 
on the modified Lyapunov function $\widetilde{\mathcal{E}}^t_{x^\star}\coloneqq\frac{1+2a_{t}\gamma\mu_\g}{2\gamma}\|z^{t}-x^\star\|^2+a_{t}^2  \big(\widetilde{\Psi}(x^{t})-\Psi^\star\big)$. This improvement transfers to the original quantities. Indeed, 
$\sqn{z^t-x^\star}\le 2\gamma \widetilde{\mathcal{E}}^t_{x^\star} =\mathcal{O}(\rho^t)$.
Moreover, by \eqref{eqala1p3}, $\|x^{t+1}-x^\star\|\leq \big(1-\frac{1}{a_\sharp}\big)\|x^{t}-x^\star\|+\frac{1}{a_\sharp}\|z^{t+1}-x^\star\|$ for $t\geq T$, where $\big(1-\frac{1}{a_\sharp}\big)^2<\rho$ (this is trivial if $a_\sharp=1$; if $a_\sharp>1$, it follows from $(1-s)^2(1+s)< 1$ for every $s\in(0,1]$, applied with $s=\frac{1}{a_\sharp}=\sqrt{2\mu_\g/L_\ff}$). 
Hence, by unrolling this recursion, $\sqn{x^t-x^\star}=\mathcal{O}(\rho^t)$, and so is $\Psi(x^t)-\Psi^\star=\widetilde{\Psi}(x^t)-\Psi^\star+\frac{\mu_\g}{2}\sqn{x^t-x^\star}$. 
Asymptotically, this matches the rate recently obtained by \citet{ush25} with a different algorithm.
\end{remark}

\subsection{Point Convergence of \algn{APGD}}\label{secpc}

When $\mu_\g>0$, Corollary~\ref{cor13} already provides the strong convergence of $(x^t)_{t\geq 0}$ to the unique minimizer, at a linear rate. We now consider the general convex case, in which only $\Psi(x^t)\rightarrow\Psi^\star$ is known from Corollary~\ref{cor13}, and complement the analysis of \algn{APGD} by establishing its point convergence. 
Historically, proving  convergence of the iterates for accelerated methods has been a notoriously difficult open problem. 
In a decisive paper, \citet{dos15} proved that if one chooses the momentum parameter in \algn{FISTA} as $a_t = \frac{t+\beta-1}{\beta}$ for some $\beta>2$, 
the sequence $(x^t)_{t\geq 0}$ converges weakly to a solution $x^\star\in\argmin \Psi$. However, the convergence of \algn{FISTA}, and even its special case \algn{AGD}, in the critical case $a_t = \frac{t+1}{2}$ remained unresolved for another decade. Recently, in late 2025, a breakthrough was achieved simultaneously by two independent teams: 
\citet{jan252} proved the convergence of \algn{AGD} (assuming $\mathcal{X}$ is finite-dimensional), while \citet{bot252} established the weak convergence of \algn{FISTA},
under the general conditions $a_{t+1}^2-a_{t+1}\leq a_t^2$ and $a_t \rightarrow \infty$. 
We show that this fundamental result also applies to \algn{APGD}. Our proof borrows from \citet{jan252} the idea of comparing $\sqn{x^t-\hat{x}}-\sqn{x^t-\check{x}}$ for two weak cluster points $\hat{x}$ and $\check{x}$ of $(x^t)_{t\geq 0}$, and relies on the following lemma, which exploits the averaging structure \eqref{eqala1p3} and will be reused for \algn{APAPC} in Section~\ref{secpointa}; related statements can be found in \citet[Lemma A.2]{bot25} and \citet{bau26}.

\begin{lemma}[Averaged sequences]\label{lemavg}
Let $(z^t)_{t\geq 0}$ be a sequence in $\mathcal{X}$, let $(a_t)_{t\geq 1}$ be a real sequence such that $a_t\geq 1$ for every $t\geq 1$ and $\sum_{t\geq 1}\frac{1}{a_t}=\infty$, and let $x^0\coloneqq z^0$ and $x^{t+1}\coloneqq \big(1-\frac{1}{a_{t+1}}\big)x^t+\frac{1}{a_{t+1}}z^{t+1}$ for every $t\geq 0$. Then, for every $t\geq 0$,
\begin{equation}
x^t=\sum_{k=0}^t w_{k,t}\,z^k,\quad\mbox{where}\quad w_{0,t}\coloneqq \prod_{j=1}^{t}\Big(1-\frac{1}{a_j}\Big)\quad\mbox{and}\quad w_{k,t}\coloneqq \frac{1}{a_k}\prod_{j=k+1}^{t}\Big(1-\frac{1}{a_j}\Big),\ \ 1\leq k\leq t,\label{eqavg}
\end{equation}
where an empty product equals one. The weights are nonnegative, $\sum_{k=0}^t w_{k,t}=1$, and $w_{k,t}\rightarrow 0$ as $t\rightarrow\infty$, for every $k\geq 0$. Consequently, $\mathrm{(i)}$ if $(z^t)_{t\geq 0}$ is bounded, then $(x^t)_{t\geq 0}$ is bounded; $\mathrm{(ii)}$ for every real sequence $(e_k)_{k\geq 0}$ such that $e_k\rightarrow 0$, we have $\sum_{k=0}^t w_{k,t}\,e_k\rightarrow 0$; $\mathrm{(iii)}$ if $(z^t)_{t\geq 0}$ is bounded and, for every pair $(\hat{x},\check{x})$ of weak cluster points of $(x^t)_{t\geq 0}$, the sequence $\big(\sqn{z^t-\hat{x}}-\sqn{z^t-\check{x}}\big)_{t\geq 0}$ converges, then $(x^t)_{t\geq 0}$ converges weakly.
\end{lemma}

\begin{proof}
The expression \eqref{eqavg} and $\sum_{k=0}^t w_{k,t}=1$ follow by induction on $t$, since $w_{k,t+1}=\big(1-\frac{1}{a_{t+1}}\big)w_{k,t}$ for $k\leq t$ and $w_{t+1,t+1}=\frac{1}{a_{t+1}}$. The weights are nonnegative since $a_j\geq 1$, and $w_{k,t}\leq \prod_{j=k+1}^{t}\big(1-\frac{1}{a_j}\big)\leq \exp\big(-\sum_{j=k+1}^t \frac{1}{a_j}\big)\rightarrow 0$, since $\sum_{j\geq 1}\frac{1}{a_j}=\infty$. Then, $\mathrm{(i)}$ is immediate and $\mathrm{(ii)}$ is the Toeplitz theorem \citep[Theorem 43.4]{kno51}. For $\mathrm{(iii)}$, let $\hat{x}$ and $\check{x}$ be two weak cluster points of $(x^t)_{t\geq 0}$, which is bounded by $\mathrm{(i)}$, and let $\phi:x\mapsto \sqn{x-\hat{x}}-\sqn{x-\check{x}}=-2\langle x,\hat{x}-\check{x}\rangle+\sqn{\hat{x}}-\sqn{\check{x}}$. By assumption, $\mathcal{B}^t\coloneqq\phi(z^t)$ converges to some $\mathcal{B}^\infty\in\mathbb{R}$; since $\phi$ is affine and the weights sum to one, \eqref{eqavg} yields $\phi(x^t)=\sum_{k=0}^t w_{k,t}\mathcal{B}^k\rightarrow \mathcal{B}^\infty$, by $\mathrm{(ii)}$ applied to $e_k = \mathcal{B}^k-\mathcal{B}^{\infty}$. Since $\phi$ is weakly continuous, $\phi(x^t)\rightarrow\phi(\hat{x})=-\sqn{\hat{x}-\check{x}}$ along a subsequence converging weakly to $\hat{x}$, and $\phi(x^t)\rightarrow\phi(\check{x})=\sqn{\hat{x}-\check{x}}$ along one converging weakly to $\check{x}$. Hence, $-\sqn{\hat{x}-\check{x}}=\mathcal{B}^\infty=\sqn{\hat{x}-\check{x}}$, so that $\hat{x}=\check{x}$, and $(x^t)_{t\geq 0}$ converges weakly to its unique weak cluster point.
\end{proof}

\begin{theorem}\label{theo1wc}
In \algn{APGD}, suppose that $\gamma\leq \frac{1}{L_\ff}$, $a_0=0$, $a_{t+1}^2-a_{t+1}\leq a_t^2$ for every $t\geq 0$, and $a_t \rightarrow \infty$.
Then $(z^t)_{t\geq 0}$ is bounded, and there exists $x^\star\in\argmin \Psi$ such that $(x^t)_{t\geq 0}$ and $(y^t)_{t\geq 0}$ both converge weakly to $x^\star$. 
\end{theorem}

\begin{proof}
For every $x^\star \in \argmin \Psi$ and $t\geq 0$, we define the Lyapunov function $\mathcal{E}^t_{x^\star}$ as in Theorem \ref{theo1}, with $\mu_\ff=\mu_\g=0$; 
then from Corollary~\ref{cor13}, we have $\frac{1}{2\gamma}\sqn{z^t -x^\star} \leq \mathcal{E}^t_{x^\star}\leq \mathcal{E}^0_{x^\star}$. Therefore, the sequence $(z^t)_{t\geq 0}$ is bounded. Moreover, since $a_t \rightarrow \infty$ and $a_{t}^2 \big(\Psi(x^{t})-\Psi^\star\big)\leq \mathcal{E}^0_{x^\star}$, we have $\Psi(x^{t})\rightarrow \Psi^\star$. Since $\Psi$ is convex and lower semicontinuous, it is weakly sequentially lower semicontinuous; consequently, every weak cluster point $x^\infty$ of $(x^t)_{t\geq 0}$ satisfies $\Psi(x^{\infty})\leq \liminf_{t \to \infty} \Psi(x^t) =\Psi^\star$, and hence is a minimizer of $\Psi$. 

Now, let $\hat{x}^\star \in \argmin \Psi$ and $\check{x}^\star \in \argmin \Psi$. Since the sequences $(\mathcal{E}^t_{\hat{x}^\star})_{t\geq 0}$ and $(\mathcal{E}^t_{\check{x}^\star})_{t\geq 0}$ are nonnegative and nonincreasing, they both converge, so that $\sqn{z^t-\hat{x}^\star}-\sqn{z^t-\check{x}^{\star}}=2\gamma\big(\mathcal{E}^t_{\hat{x}^\star}-\mathcal{E}^t_{\check{x}^\star}\big)$ converges. In addition, $a_t\leq t$ for every $t\geq 0$ by \eqref{eqagrow}, so that $\sum_{t\geq 1}\frac{1}{a_t}=\infty$. The conditions of Lemma~\ref{lemavg}$\mathrm{(iii)}$ are therefore met, and $(x^t)_{t\geq 0}$ converges weakly to some $x^\star$, which is a minimizer of $\Psi$ by the first part of the proof. 

Using \eqref{eqala1}, we have $\|y^t-x^t\|=\frac{1}{a_{t+1}}\|x^t-z^t\|\rightarrow 0$, because $a_t\rightarrow \infty$ and $(x^t-z^t)_{t\geq 0}$ is bounded. Hence, $(y^t)_{t\geq 0}$ also converges weakly to $x^\star$.
\end{proof}

\subsection{Proof of Lemma~\ref{lemman1}}\label{secprooflem1}

Let $x^\star\in\argmin \Psi$ 
and $t\geq 0$. We define the Bregman distances 
\begin{align*}
&D_\ff: x\in\mathcal{X}\mapsto \ff(x)-\ff(x^\star)-\langle x-x^\star,\nabla \ff(x^\star)\rangle \geq 0,\\
&D_\g: x\in\mathcal{X}\mapsto \g(x)-\g(x^\star)-\langle x-x^\star,\tilde{\nabla} \g(x^\star)\rangle \geq 0, 
\end{align*}
where  $\tilde{\nabla}\g(x^\star)\coloneqq -\nabla \ff(x^\star)$. We want to express
\begin{align*}
\big\langle \tilde{\nabla}\g(z^{t+1})+\nabla \ff(y^t),z^{t+1}-x^\star\big\rangle&=\big\langle \tilde{\nabla}\g(z^{t+1})- \tilde{\nabla}\g(x^\star),z^{t+1}-x^\star\big\rangle\\
&+\big\langle \nabla \ff(y^t)-\nabla \ff(x^\star),z^{t+1}-x^\star\big\rangle
\end{align*}
using $D_\ff$ and $D_\g$. 

\textbf{Step 1: Bounding the nonsmooth term.} By $\mu_\g$-strong convexity of $\g$, we have $\g(x^\star)\geq \g(z^{t+1})+ \langle \tilde{\nabla}\g(z^{t+1}),x^\star-z^{t+1}\rangle + \frac{\mu_\g}{2}\sqn{z^{t+1}-x^\star}$, so that 
$D_\g(z^{t+1}) \leq \langle \tilde{\nabla}\g(z^{t+1})-\tilde{\nabla} \g(x^\star),z^{t+1}-x^\star\rangle - \frac{\mu_\g}{2}\sqn{z^{t+1}-x^\star}$, and
 \begin{align}
&-a_{t+1} \big\langle \tilde{\nabla}\g(z^{t+1})-\tilde{\nabla}\g(x^\star),z^{t+1}-x^\star\big\rangle\notag\\
&\leq -a_{t+1} D_\g(z^{t+1})-\frac{a_{t+1}\mu_\g}{2}\|z^{t+1}-x^\star\|^2\notag\\
& \leq -a_{t+1}^2 D_\g(x^{t+1})+(a_{t+1}^2-a_{t+1})D_\g(x^{t})-\frac{a_{t+1}\mu_\g}{2}\|z^{t+1}-x^\star\|^2.\label{eqlem1k1}
\end{align}
The last inequality follows from the convexity of $D_\g$: we have $D_\g(x^{t+1}) \leq \big(1-\frac{1}{a_{t+1}}\big) D_\g(x^{t})+\frac{1}{a_{t+1}} D_\g(z^{t+1})$, so that $D_\g(z^{t+1})\geq a_{t+1} D_\g(x^{t+1})-(a_{t+1}-1)D_\g(x^{t})$.

\textbf{Step 2: Expanding the gradient term.} We expand $\big\langle \nabla \ff(y^t)-\nabla \ff(x^\star),z^{t+1}-x^\star\big\rangle$
into inner products of the form $\big\langle \nabla \ff(y^t)-\nabla \ff(x^\star),y^t-X\big\rangle$, where $X$ is $x^\star$, $x^t$, or $x^{t+1}$, and subsequently use convexity and smoothness at $y^t$ to express them using $D_\ff(y^t)$, $D_\ff(x^t)$, and $D_\ff(x^{t+1})$.
From $y^t=\frac{1}{a_{t+1}} z^t+(1-\frac{1}{a_{t+1}}) x^t$ and $x^{t+1}-y^t=\frac{1}{a_{t+1}}(z^{t+1}-z^t)$, we obtain 
$z^{t+1}-x^\star =(z^{t+1}-z^t)+(z^t-x^\star)= -a_{t+1}(y^t-x^{t+1})+(y^t-x^\star)+(a_{t+1}-1)(y^t - x^t)$. Hence,
\begin{align}
\big\langle \nabla \ff(y^t)-\nabla \ff(x^\star),z^{t+1}-x^\star\big\rangle&=
-a_{t+1}\big\langle \nabla \ff(y^t)-\nabla \ff(x^\star),y^t-x^{t+1}\big\rangle+\big\langle \nabla \ff(y^t)-\nabla \ff(x^\star),y^t-x^\star\big\rangle\notag\\
&\quad+(a_{t+1}-1)\big\langle \nabla \ff(y^t)-\nabla \ff(x^\star),y^t-x^t\big\rangle.\label{eqwy1}
\end{align}
\textbf{Step 2-a.} Since $\ff$ is $L_\ff$-smooth, we have
$\ff(x^{t+1}) \leq \ff(y^t)+\langle \nabla \ff(y^t), x^{t+1}-y^t\rangle +\frac{L_\ff}{2}\|x^{t+1}-y^t\|^2$. Subtracting $\ff(x^\star)+\langle \nabla \ff(x^\star), x^{t+1}-x^\star\rangle$ from both sides and rearranging, we obtain
\begin{equation}
\langle \nabla \ff(y^t)-\nabla \ff(x^\star), y^t-x^{t+1}\rangle\leq D_\ff(y^t)-D_\ff(x^{t+1})+\frac{L_\ff}{2}\|x^{t+1}-y^t\|^2.\label{eqstep2a}
\end{equation}
\textbf{Step 2-b.} Since $\ff$ is $\mu_\ff$-strongly convex and $L_\ff$-smooth, we have \citep[Theorem 4]{tay17}
$\ff(x^t) \geq \ff(y^t) + \langle \nabla \ff(y^t), x^t - y^t \rangle + \frac{\mu_\ff}{2} \sqn{x^t - y^t}+\sigma_t$. Subtracting $\ff(x^\star)+\langle \nabla \ff(x^\star), x^{t}-x^\star\rangle$ from both sides and rearranging, we obtain
\begin{equation}
\langle \nabla \ff(y^t)-\nabla \ff(x^\star), y^t-x^{t}\rangle \geq D_\ff(y^t)-D_\ff(x^{t}) +\frac{\mu_\ff}{2}\sqn{y^t-x^t}+\sigma_t.\label{eqstep2b}
\end{equation}
\textbf{Step 2-c.} Using the same derivations as in Step 2-b, with $x^t$ replaced by $x^\star$, we obtain
\begin{equation}
\left\langle \nabla \ff(y^t)-\nabla \ff(x^\star),y^{t}-x^\star\right\rangle\geq D_\ff(y^t)+\frac{\mu_\ff}{2}\sqn{y^t-x^\star}+\sigma^\star_t.\label{eqstep2c}
\end{equation}
Therefore, multiplying \eqref{eqwy1} by $-a_{t+1}$ and using \eqref{eqstep2a}--\eqref{eqstep2c}, together with $a_{t+1}-1\geq 0$ and $\|x^{t+1}-y^t\|^2=\frac{1}{a_{t+1}^2}\|z^{t+1}-z^t\|^2$, 
the $D_\ff(y^t)$ terms cancel, and
we have
\begin{align*}
-a_{t+1}\left\langle \nabla \ff(y^t)-\nabla \ff(x^\star),z^{t+1}-x^\star\right\rangle
&\le -a_{t+1}^2 D_\ff(x^{t+1})+(a_{t+1}^2-a_{t+1})D_\ff(x^{t})+{\textstyle\frac{L_\ff}{2}}\|z^{t+1}-z^t\|^2\\
&\quad-{\textstyle\frac{\mu_\ff}{2}}\Big(a_{t+1}\|y^t -x^\star\|^2+(a_{t+1}^2-a_{t+1})\|y^t-x^t\|^2\Big)\\
&\quad-a_{t+1}\sigma^\star_t-(a_{t+1}^2-a_{t+1})\sigma_t.
\end{align*}
Furthermore, by Jensen's inequality applied to the relation 
$\frac{1}{a_{t+1}}(z^t-x^\star) = \frac{1}{a_{t+1}} (y^t-x^\star)+\frac{a_{t+1}-1}{a_{t+1}}(y^t - x^t)$, we have
$\sqn{z^t-x^\star} \leq a_{t+1}\sqn{y^t-x^\star}+(a_{t+1}^2-a_{t+1})\sqn{y^t - x^t}$. Hence, 
\begin{align}
-a_{t+1}\left\langle \nabla \ff(y^t)-\nabla \ff(x^\star),z^{t+1}-x^\star\right\rangle&\leq -a_{t+1}^2 D_\ff(x^{t+1})+(a_{t+1}^2-a_{t+1})D_\ff(x^{t})+{\textstyle\frac{L_\ff}{2}}\|z^{t+1}-z^t\|^2\notag\\
&\quad-{\textstyle\frac{\mu_\ff}{2}}\|z^t -x^\star\|^2-a_{t+1}\sigma^\star_t-(a_{t+1}^2-a_{t+1})\sigma_t. \label{eqlem1k2} 
\end{align}
Combining \eqref{eqlem1k1} and \eqref{eqlem1k2}, and noting that $D_\ff+D_\g=\ff-\ff(x^\star)+\g-\g(x^\star)=\Psi-\Psi^\star$, we obtain the stated inequality.

\section{The  Accelerated Proximal Alternating Predictor--Corrector Algorithm (\algn{APAPC})}\label{sec2}

We now consider  Problem \eqref{eq1} under the specific assumption that $\g\coloneqq \frac{\mu_\g}{2}\|\cdot\|^2$ with $\mu_\g\ge 0$: 
\begin{equation}
\minimize_{x\in\mathcal{X}} \;\Psi(x)\coloneqq \ff(x)+\frac{\mu_\g}{2}\|x\|^2+\h(\K x). \label{eq1c}
\end{equation}
We assume that $\mu_\ff=0$, without loss of generality: if $\ff$ is $\mu_\ff$-strongly convex, the transfer of strong convexity described in Remark~\ref{remmuf} yields a problem of the same form, with 
$\mu_\g+\mu_\ff$ in place of $\mu_\g$. 
Thus, Problem~\eqref{eq1c} consists of minimizing the sum of a smooth, possibly strongly convex, function $\ff+\frac{\mu_\g}{2}\|\cdot\|^2$ and a nonsmooth composite term $\h\circ\K$; 
this is the problem class addressed by \algn{PAPC}. As the purpose of acceleration is to improve the dependence on the condition number $L_\ff/\mu_\g$ when it is large, we also assume that $L_\ff\geq \mu_\g$; otherwise, $a_\sharp=1$ in Corollary~\ref{cor13} and momentum is disabled. 
Problem~\eqref{eq1c} is more general than Problem \eqref{eq1p}: setting $\mu_\g=0$ and $\K=\Id$ recovers the minimization of $\ff+\h$, as detailed in Remark \ref{remkid}. 
The corresponding dual problem is 
\begin{equation}
\minimize_{u\in\mathcal{U}} \;
\left(\ff+\frac{\mu_\g}{2}\|\cdot\|^2\right)^*(-\K^*u)+\hc(u). \label{eq1cd}
\end{equation}

\begin{algorithm}[t]
	\caption{Accelerated Proximal Alternating Predictor--Corrector algorithm (\algn{APAPC}) 
	}
	\label{alg2}
	\begin{algorithmic}[1]
		\STATE \textbf{input:}  primal stepsize $\gamma>0$, dual stepsize $\tau>0$, parameter sequence $(a_t)_{t\geq 0}$ with $a_t\geq 1$ $\forall t\geq 1$,  initial primal estimate $x^0=z^0\in\mathcal{X}$, initial dual estimate $u^0=v^0\in\mathcal{U}$		
		\FOR{$t=0,1,\ldots$}
	\STATE	$y^{t}  \coloneqq \left(1-\frac{1}{a_{t+1}}\right) x^{t}+\frac{1}{a_{t+1}} z^{t}$\label{eqapapcs1k}
\STATE $\hat{z}^t\coloneqq (1+a_{t+1}\gamma \mu_\g)^{-1}\left(z^t-a_{t+1}\gamma\nabla \ff(y^t)-a_{t+1}\gamma\K^* v^t\right)$\label{eqapapcc1k} 
\STATE $v^{t+1}\coloneqq \mathrm{prox}_{\frac{\tau}{a_{t+1}} \hc} \Big(v^t+{\textstyle\frac{\tau}{a_{t+1}}} \K \hat{z}^t\Big)$ \label{eqapapcc2k} %-b)\\
\STATE $z^{t+1}\coloneqq (1+a_{t+1}\gamma \mu_\g)^{-1}\left(z^t-a_{t+1}\gamma\nabla \ff(y^t)-a_{t+1}\gamma\K^* v^{t+1}\right)$ \label{eqapapcc3k} 
\STATE $x^{t+1}  \coloneqq \left(1-\frac{1}{a_{t+1}}\right) x^{t}+\frac{1}{a_{t+1}} z^{t+1}$ \label{eqapapcc4k}  \\
\STATE $u^{t+1}  \coloneqq \left(1-\frac{1}{a_{t+1}}\right) u^{t}+\frac{1}{a_{t+1}} v^{t+1} \quad$ (optional, used only in the analysis) \label{eqapapcc5k} 
		\ENDFOR
\end{algorithmic}
\end{algorithm}

To solve Problem~\eqref{eq1c}, we introduce the \emph{Accelerated Proximal Alternating Predictor--Corrector} algorithm (\algn{APAPC}), given in Algorithm \ref{alg2}, which proceeds as follows.
The three core Steps \ref{eqapapcc1k}--\ref{eqapapcc3k} mirror the \algn{PAPC} mechanism for updating $z^t$ and $v^t$, with specific stepsizes and the key distinction that $\nabla \ff$ is evaluated at $y^t$ rather than at $z^t$. Step \ref{eqapapcc1k} predicts the updated primal variable $\hat{z}^t$ using $(z^t,v^t)$, Step \ref{eqapapcc2k} performs the dual update using $(\hat{z}^t,v^t)$, and Step \ref{eqapapcc3k} corrects the primal variable using the updated dual variable $v^{t+1}$.
The final Steps \ref{eqapapcc4k}--\ref{eqapapcc5k} form the primal and dual estimates $x^t$ and $u^t$ as weighted averages of the past iterates $z^k$ and $v^k$, $k=0,\ldots,t$. 
The intermediate point $y^t$ is also a weighted average of the past $z^k$, mirroring the momentum  architecture of \algn{APGD}. 
 
 When $a_{t}\equiv 1$ and $\mu_\g=0$, we have $x^t=y^t=z^t$ and $v^t=u^t$, so that \algn{APAPC} reduces exactly to \algn{PAPC} (see \eqref{eqpapc}), which justifies its name. 
 Adding the function $\g \coloneqq \frac{\mu_\g}{2}\|\cdot\|^2$, whose proximity operator is a simple scaling,  to Problem \eqref{eq1c}  does not alter the forward--backward structure of \algn{PAPC} \citep[Section 4.1]{con23}. 
However, this term, when $\mu_\g>0$, is crucial
for achieving accelerated linear convergence, as demonstrated for \algn{APGD} in Section \ref{secne}.

Thus, \algn{APAPC} addresses the challenge of accelerating \algn{PAPC} 
by combining the decoupled momentum architecture
 of \algn{APGD} (exhibited in Section \ref{secne}) with the forward--backward splitting structure of \algn{PAPC}.
In this correspondence, the nonsmooth function $\g$ of Section~\ref{secne} is replaced by the composite term $\h\circ\K$, which is handled through the dual variable, while the quadratic $\g$ of Problem~\eqref{eq1c} plays the role of the strongly convex part of the objective and enters the algorithm only through the scaling factor $(1+a_{t+1}\gamma\mu_\g)^{-1}$; Remark~\ref{remkid} makes this correspondence exact.
 The result is a fully split, single-loop algorithm that accommodates the composite term $\h\circ\K$ 
while preserving optimal acceleration with respect to the smooth component $\ff$.

In the remainder of this section, we establish accelerated $\mathcal{O}(1/t^2)$ convergence of \algn{APAPC} in three settings where the dual problem is strongly convex. If, in addition, $\mu_\g>0$, then the primal problem is also strongly convex and \algn{APAPC} achieves accelerated linear convergence. We mention that the symmetric case where only the primal problem is strongly convex (i.e., $\mu_\g>0$ and $\h\circ\K$ is arbitrary) can be handled by disabling momentum and using a \emph{decreasing} parameter $a_t$ in Steps \ref{eqapapcc1k}--\ref{eqapapcc3k} \citep{con22,con25smpm}. We do not consider this case here, as our focus is on Nesterov acceleration of the smooth term, which relies on strong convexity of the dual problem.

We assume $a_t \geq 1$ for every $t\geq 1$ and $a_0 \geq 0$. 
We set $a_1=1$ in our convergence results, ensuring that the first iteration of \algn{APAPC} coincides with a \algn{PAPC} iteration. 
In particular, $x^1=z^1$, and $u^1=v^1$ belongs to the domain of $\hc$, as do all subsequent iterates $u^t$.

\subsection{General Lyapunov Analysis}

Following the approach used for \algn{APGD} in Theorem~\ref{theo1}, 
we first establish a Lyapunov inequality for 
\algn{APAPC}. In the purely primal setting of Section \ref{secne}, convergence rates were characterized by the primal gap $\Psi(x^t)-\Psi(x^\star)$. In the primal--dual setting, we instead consider the Lagrangian 
\begin{align*}
\mathcal{L}: (x,u)\in\mathcal{X}\times \mathcal{U} \mapsto (\ff+\g)(x)+\langle \K x,u\rangle - \hc(u).
\end{align*}
For every saddle-point solution $(x^\star,u^\star)$   to \eqref{eq2}, 
we define the Lagrangian gap as
\begin{align*}
\mathcal{G}_{x^\star,u^\star}(x,u)\coloneqq 
\mathcal{L} (x,u^\star)-\mathcal{L} (x^\star,u)
\geq \frac{\mu_\g}{2}\sqn{x-x^\star}+\frac{\mu_\hc}{2}\sqn{u-u^\star},\quad \forall (x,u)\in\mathcal{X}\times \mathcal{U}.
\end{align*}
When $\h=0$, the dual variable can be fixed to zero, and this gap reduces to the primal objective gap: $\mathcal{L}(x,0)-\mathcal{L}(x^\star,0) = (\ff+\g)(x)-(\ff+\g)(x^\star) = \Psi(x)-\Psi(x^\star)$. Thus, the Lagrangian gap provides a natural extension of the primal objective gap to the primal--dual setting. However, the Lagrangian gap alone may not serve as a fully informative optimality measure in all problem instances, as the condition $\mathcal{G}_{x^\star,u^\star}(x,u)=0$ does not generally guarantee that $(x,u)$ is a primal--dual solution pair. This limitation underscores the importance of complementing the convergence rate analysis with a study of the iterate convergence, which we undertake in Section \ref{secpointa}.

To state our core single-iteration inequality for \algn{APAPC}, we define, for every $t\geq 0$, the subgradient $\tilde{\nabla}\hc(v^{t+1}) \in\partial\hc(v^{t+1})$ satisfying $v^{t+1}+\frac{\tau}{a_{t+1}}\tilde{\nabla}\hc(v^{t+1})=v^t+{\textstyle\frac{\tau}{a_{t+1}}} \K \hat{z}^t$.
We also define, for every $t\geq 0$, the block linear operator on $\mathcal{X}\times\mathcal{U}$
\begin{equation}
\Q^t\coloneqq \begin{bmatrix}
\frac{1}{\gamma} \Id & \ze \\
\ze &\Q^t_{\mathcal{U}}
\end{bmatrix}, \quad \mbox{with } \Q^t_{\mathcal{U}}\eqdef\frac{a_t^2}{\tau} \Id-\frac{a_t^2\gamma}{1+a_t\gamma \mu_\g}\K\K^*,\label{eqblo}
\end{equation}
where we adopt matrix-vector notation for the variables and linear operators on $\mathcal{X}\times\mathcal{U}$. The operator $\Q^t$ is self-adjoint. If $\gamma\tau\|\K\K^*\|=\gamma\tau\|\K\|^2\leq 1+a_t\gamma\mu_\g$, then 
$\Q^t_{\mathcal{U}}$ and $\Q^t$ are positive semidefinite (positive definite if the inequality is strict), 
with $\|\Q^t_{\mathcal{U}}\|\leq \frac{a_t^2}{\tau}$.
In that case, $\langle \cdot,\cdot \rangle_{\Q^t}\coloneqq\langle \cdot, \Q^t \cdot\rangle$ is a semi-inner product on $\mathcal{X}\times\mathcal{U}$, with associated seminorm $\|\cdot\|_{\Q^t}$, and, for every $(x,u)\in\mathcal{X}\times\mathcal{U}$,
\begin{equation}
\frac{1}{2}\left\|\begin{bmatrix}
x \\
u
\end{bmatrix}\right\|_{\Q^t}^2=\frac{1}{2\gamma}\sqn{x}+\frac{a_t^2}{2\tau}\sqn{u}-\frac{a_t^2\gamma}{2+2a_t\gamma\mu_\g}\sqn{\K^*u}\geq 0.\label{eqqnonneg}
\end{equation}
This metric is motivated by the fact that \algn{PAPC} can be cast as a preconditioned forward--backward algorithm in $\mathcal{X}\times\mathcal{U}$ endowed with it \citep{com14,con23}. The condition $\gamma\tau\|\K\|^2\leq 1+a_t\gamma\mu_\g$ will ensure that the corresponding terms in \eqref{eqinea1} below and in the Lyapunov functions defined in the sequel are nonnegative.

\begin{theorem}[Single-iteration progress of \algn{APAPC}]\label{theo2}
Let $(x^\star,u^\star)$
 be a solution of \eqref{eq2} and let $t\geq 0$ be such that $\gamma\tau\|\K\|^2\leq 1+a_{t+1}\gamma\mu_\g$. 
Then the iterates of \algn{APAPC} satisfy
\begin{align}
&\frac{1+a_{t+1}\gamma\mu_\g}{2\gamma}\sqn{z^{t+1}-x^\star}+\frac{a_{t+1}^2+a_{t+1}\tau\mu_\hc}{2\tau}\sqn{v^{t+1}-u^{\star}}  -\frac{a_{t+1}^2\gamma}{2+2a_{t+1}\gamma \mu_\g}\sqn{\K^*(v^{t+1}-u^{\star})}\notag\\
&\quad +a_{t+1}^2 \mathcal{G}_{x^\star,u^\star}(x^{t+1},u^{t+1})
\notag\\
&\leq\frac{1}{2\gamma}\sqn{z^{t}-x^\star}+\frac{a_{t+1}^2}{2\tau}\sqn{v^{t}-u^{\star}}  -\frac{a_{t+1}^2\gamma}{2+2a_{t+1}\gamma \mu_\g}\sqn{\K^*(v^{t}-u^{\star})}\label{eqinea1}\\
&\quad -\left(\frac{1}{2\gamma}-\frac{L_\ff}{2}\right)\sqn{z^{t+1}-z^t } 
- \frac{\tau}{2}\sqn{ \K z^{t+1}-\tilde{\nabla}\hc(v^{t+1})}\notag\\
&\quad+(a_{t+1}^2-a_{t+1})\mathcal{G}_{x^\star,u^\star}(x^{t},u^{t}) -\frac{a_{t+1}}{2 L_\ff} \sqn{\nabla \ff(y^t) - \nabla \ff(x^\star) }
-\frac{a_{t+1}^2-a_{t+1}}{2 L_\ff} \sqn{\nabla \ff(x^t) - \nabla \ff(y^t) }.\notag
\end{align}
\end{theorem}

\begin{proof}
Let $(x^\star,u^\star)$ be a solution of \eqref{eq2}. 
We conduct our analysis in the product space $\mathcal{X}\times\mathcal{U}$, with the operators $\Q^t$ and the semi-inner products defined in \eqref{eqblo}--\eqref{eqqnonneg}.
We define $\tilde{\nabla}\hc(u^\star)\coloneqq \K x^\star$. Let $t\geq 0$. 
From Steps \ref{eqapapcc1k} and \ref{eqapapcc3k}, we have
$(1+a_{t+1}\gamma\mu_\g)(z^{t+1}-\hat{z}^t)=-a_{t+1}\gamma \K^*(v^{t+1}-v^t)$. 
Substituting this relation into the dual update in Step \ref{eqapapcc2k}, written as $v^{t+1}-v^t=\frac{\tau}{a_{t+1}}\big(\K \hat{z}^t-\tilde{\nabla}\hc(v^{t+1})\big)$, yields
\begin{equation}
\Q^{t+1}_{\mathcal{U}}(v^{t+1}-v^t)=a_{t+1}\K \hat{z}^t-a_{t+1}\tilde{\nabla}\hc(v^{t+1})+a_{t+1}\K(z^{t+1}-\hat{z}^t)=a_{t+1}\big(\K z^{t+1}-\tilde{\nabla}\hc(v^{t+1})\big).\label{eqqvv}
\end{equation}
Therefore, 
\begin{align*}
\Q^{t+1}\begin{bmatrix}
z^{t+1}-z^t \\
v^{t+1}-v^t
\end{bmatrix}&=a_{t+1}\begin{bmatrix}
-\mu_\g z^{t+1}-\nabla \ff(y^t)-\K^* v^{t+1} \\
\K z^{t+1}-\tilde{\nabla}\hc(v^{t+1})
\end{bmatrix} \\
&=a_{t+1}\begin{bmatrix}
-\mu_\g(z^{t+1}-x^\star)-\nabla \ff(y^t)+\nabla \ff(x^\star)-\K^* (v^{t+1}-u^\star) \\
\K (z^{t+1}-x^\star)-\tilde{\nabla}\hc(v^{t+1})+\tilde{\nabla}\hc(u^\star)
\end{bmatrix},
\end{align*}
where we used the optimality condition $0=\mu_\g x^\star + \nabla \ff( x^\star) + \K^*u^\star$. Hence, expanding the norm gives
\begin{align}
\frac{1}{2}\left\|\begin{bmatrix}
z^{t+1}-x^\star \\
v^{t+1}-u^{\star}
\end{bmatrix}\right\|_{\Q^{t+1}}^2&=\frac{1}{2}\left\|\begin{bmatrix}
z^t-x^\star \\
v^t-u^{\star}
\end{bmatrix}\right\|_{\Q^{t+1}}^2-\frac{1}{2}\left\|\begin{bmatrix}
z^{t+1}-z^t \\
v^{t+1}-v^t
\end{bmatrix}\right\|_{\Q^{t+1}}^2+\left\langle\begin{bmatrix}
z^{t+1}-z^t \\
v^{t+1}-v^t
\end{bmatrix},\begin{bmatrix}
z^{t+1}-x^\star \\
v^{t+1}-u^{\star}
\end{bmatrix}\right\rangle_{\Q^{t+1}}\notag\\
&= \frac{1}{2}\left\|\begin{bmatrix}
z^t-x^\star \\
v^t-u^{\star}
\end{bmatrix}\right\|_{\Q^{t+1}}^2-\frac{1}{2\gamma}\sqn{z^{t+1}-z^t}- \frac{1}{2}\left\langle v^{t+1}-v^t, \Q^{t+1}_{\mathcal{U}}(v^{t+1}-v^t)\right\rangle\notag\\
&\quad-a_{t+1}\mu_\g\sqn{z^{t+1}-x^\star}-a_{t+1}\left\langle \nabla \ff(y^t)-\nabla \ff(x^\star),z^{t+1}-x^\star\right\rangle\label{eqkk1}\\
&\quad-a_{t+1}\left\langle \tilde{\nabla}\hc(v^{t+1})-\tilde{\nabla}\hc(u^\star),v^{t+1}-u^\star\right\rangle,\notag
\end{align}
where the coupling terms $-a_{t+1}\langle \K^* (v^{t+1}-u^\star),z^{t+1}-x^\star\rangle+a_{t+1}\langle \K (z^{t+1}-x^\star),v^{t+1}-u^\star\rangle$, which cancel out, have been omitted.
Moreover, since $\gamma\tau\|\K\|^2\leq 1+a_{t+1}\gamma\mu_\g$, $\Q^{t+1}_{\mathcal{U}}$ is positive semidefinite with $\|\Q^{t+1}_{\mathcal{U}}\|\leq \frac{a_{t+1}^2}{\tau}$, as noted after \eqref{eqblo}. Hence,
\begin{align*}
\left\langle v^{t+1}-v^t, \Q^{t+1}_{\mathcal{U}}(v^{t+1}-v^t)\right\rangle &\ge \frac{1}{\|\Q^{t+1}_{\mathcal{U}}\|} \sqn{\Q^{t+1}_{\mathcal{U}}(v^{t+1}-v^t)}\\
&\ge \frac{\tau}{a_{t+1}^2}\sqn{\Q^{t+1}_{\mathcal{U}}(v^{t+1}-v^t)}=\tau\sqn{ \K z^{t+1}-\tilde{\nabla}\hc(v^{t+1})}.
\end{align*}
We note that if $\|\Q^{t+1}_{\mathcal{U}}\|=0$, this inequality holds trivially, as both sides vanish and $\K z^{t+1}-\tilde{\nabla}\hc(v^{t+1})=0$.

We introduce the Bregman distances 
\begin{align*}
&D_\ff: x\in\mathcal{X}\mapsto \ff(x)-\ff(x^\star)-\langle x-x^\star,\nabla \ff(x^\star)\rangle \geq 0,\\
&D_\g: x\in\mathcal{X}\mapsto \g(x)-\g(x^\star)-\langle x-x^\star,\mu_\g x^\star\rangle=\frac{\mu_\g}{2}\sqn{x-x^\star} \geq 0, \\
&D_\hc: u\in\mathcal{U}\mapsto \hc(u)-\hc(u^\star)-\langle u-u^\star,\tilde{\nabla}\hc(u^\star)\rangle \geq 0.
\end{align*}
Then, the Lagrangian gap can be expressed as
\begin{align}
\mathcal{G}_{x^\star,u^\star}(x,u)=\mathcal{L} (x,u^\star)-\mathcal{L} (x^\star,u)=(D_\ff+D_\g)(x)+D_\hc(u),\quad \forall (x,u)\in\mathcal{X}\times \mathcal{U}.\label{eqgapw}
\end{align}
Proceeding as in the proof of Lemma~\ref{lemman1}, we obtain
\begin{align*}
&-a_{t+1}\mu_\g\sqn{z^{t+1}-x^\star}-a_{t+1}\left\langle \nabla \ff(y^t)-\nabla \ff(x^\star),z^{t+1}-x^\star\right\rangle\\
& \leq -a_{t+1}^2 (D_\ff+D_\g)(x^{t+1})+(a_{t+1}^2-a_{t+1})(D_\ff+D_\g)(x^{t})-\frac{a_{t+1}\mu_\g}{2}\sqn{z^{t+1}-x^\star}+\frac{L_\ff}{2}\sqn{z^{t+1}-z^t}\\
&\quad  -\frac{a_{t+1}}{2 L_\ff} \sqn{\nabla \ff(y^t) - \nabla \ff(x^\star) } -\frac{a_{t+1}^2-a_{t+1}}{2 L_\ff} \sqn{\nabla \ff(x^t) - \nabla \ff(y^t) }.
 \end{align*}
Similarly, by the $\mu_\hc$-strong convexity of $\hc$, we have $\hc(u^\star)\geq \hc(v^{t+1})+ \langle \tilde{\nabla}\hc(v^{t+1}),u^\star-v^{t+1}\rangle + \frac{\mu_\hc}{2}\sqn{v^{t+1}-u^\star}$, so that 
$D_\hc(v^{t+1}) \leq \langle \tilde{\nabla}\hc(v^{t+1})-\tilde{\nabla} \hc(u^\star),v^{t+1}-u^\star\rangle - \frac{\mu_\hc}{2}\sqn{v^{t+1}-u^\star}$. Moreover, by the convexity of $D_\hc$, we have $D_\hc(u^{t+1}) \leq \big(1-\frac{1}{a_{t+1}}\big) D_\hc(u^{t})+\frac{1}{a_{t+1}} D_\hc(v^{t+1})$. Hence,
 \begin{align*}
&-a_{t+1} \big\langle \tilde{\nabla}\hc(v^{t+1})-\tilde{\nabla}\hc(u^\star),v^{t+1}-u^\star\big\rangle\notag\\
&\leq -a_{t+1} D_\hc(v^{t+1})-\frac{a_{t+1}\mu_\hc}{2}\sqn{v^{t+1}-u^\star}\notag\\
& \leq -a_{t+1}^2 D_\hc(u^{t+1})+(a_{t+1}^2-a_{t+1})D_\hc(u^{t})-\frac{a_{t+1}\mu_\hc}{2}\sqn{v^{t+1}-u^\star}.
\end{align*}
Substituting these bounds into \eqref{eqkk1}, expanding the squared norms using \eqref{eqqnonneg}, rearranging the terms, and substituting the Lagrangian gap \eqref{eqgapw} for the Bregman distances yields the inequality \eqref{eqinea1}.
\end{proof}

\subsection{Accelerated Convergence with $\h$ Smooth}\label{sechs}

In this section, we consider the case where $\mu_\hc>0$. Equivalently, 
$\h$ is $L_\h$-smooth with  $L_\h = 1/\mu_\hc$.
In this regime, the dual problem \eqref{eq1cd} is strongly convex, ensuring that its solution $u^\star$ is unique. 
The whole objective $\ff+\h\circ\K$ is smooth, and we explain in Remark~\ref{remhsmooth} 
why a primal--dual treatment is nonetheless of interest.
As a consequence of Theorem \ref{theo2}, we obtain the following convergence guarantees.

\begin{theorem}[Accelerated convergence of \algn{APAPC} with $\mu_\hc>0$]\label{corcase1a}
Suppose that $\mu_\hc>0$  and
let $(x^\star,u^\star)$ 
 be a solution of \eqref{eq2}.  In \algn{APAPC}, 
 define the Lyapunov function 
\begin{align*}
\mathcal{E}^t_{x^\star}&\coloneqq \frac{1+a_{t}\gamma\mu_\g}{2\gamma}\sqn{z^{t}-x^\star}+\frac{a_{t}^2+a_{t}\tau\mu_\hc}{2\tau}\sqn{v^{t}-u^{\star}}  -\frac{a_{t}^2\gamma}{2+2a_{t}\gamma \mu_\g}\sqn{\K^*(v^{t}-u^{\star})}\\
&\quad +a_{t}^2 
\mathcal{G}_{x^\star,u^\star}(x^t,u^t),\quad\forall t\geq 0.
\end{align*}
Suppose that $u^0\in\mathrm{dom}\,\hc$, $\gamma\leq \frac{1}{L_\ff}$, $a_0>0$, $\gamma\tau\|\K\|^2\leq 1+a_0\gamma\mu_\g$, and $a_t\leq a_{t+1}\leq \sqrt{a_t^2+a_t\tau\mu_\hc}$ for every $t\geq 0$ (since $a_1\geq 1$ is required, the latter condition imposes $a_0\geq \frac{\sqrt{\tau^2\mu^2_\hc+4}-\tau\mu_\hc}{2}$).
Then 
\begin{equation*}
\mathcal{E}^{t+1}_{x^\star} \leq \max\left(\frac{1}{1+a_{t}\gamma\mu_\g},\frac{a_{t+1}^2}{a_t^2+a_t\tau\mu_\hc},
\frac{a_{t+1}^2-a_{t+1}}{a_t^2}\right)\mathcal{E}^t_{x^\star} ,\quad\forall t\geq 0.
\end{equation*}
Consequently, if, in addition, $a_{t+1}^2-a_{t+1}\leq a_t^2$ for every $t\geq 0$, then 
\begin{equation*}
\mathcal{E}^{t}_{x^\star} \leq \mathcal{E}^0_{x^\star} ,\quad\forall t\geq 0,
\end{equation*}
so that the Lagrangian gap is bounded by
\begin{equation*}
\mathcal{G}_{x^\star,u^\star}(x^t,u^t)\leq \frac{\mathcal{E}^0_{x^\star} }{a_t^2} ,\quad\forall t\geq 0.
\end{equation*}
In particular, by choosing 
$a_0 \coloneqq \frac{\sqrt{\tau^2\mu^2_\hc+4}-\tau\mu_\hc}{2}$, $a_1=1$, and 
\begin{equation*}
a_{t+1}\coloneqq \min\left({\textstyle \sqrt{a_t^2+a_t\tau\mu_\hc},\frac{1+\sqrt{1+4a_t^2}}{2}}
\right),\quad\forall t\geq 1,
\end{equation*}
which 
grows 
as $a_t=\Omega\left( \frac{t}{2}\min(1,\tau\mu_\hc)\right)$, 
\algn{APAPC} achieves an accelerated sublinear convergence rate of
\begin{equation*}
\mathcal{G}_{x^\star,u^\star}(x^t,u^t)= \mathcal{O}\left(\frac{\mathcal{E}^0_{x^\star} }{\min(1,\tau\mu_\hc)^2 t^2}\right).
\end{equation*}
This implies $\sqn{u^t-u^\star}=\mathcal{O}(1/t^2)$ as well. 
\end{theorem}

\begin{proof}
Since $(a_t)_{t\geq 0}$ is nondecreasing and $\gamma\tau\|\K\|^2\leq 1+a_0\gamma\mu_\g$, Theorem~\ref{theo2} applies for every $t\geq 0$. Since $\gamma\leq \frac{1}{L_\ff}$, the terms $-\big(\frac{1}{2\gamma}-\frac{L_\ff}{2}\big)\sqn{z^{t+1}-z^t }$, $- \frac{\tau}{2}\sqn{ \K z^{t+1}-\tilde{\nabla}\hc(v^{t+1})}$ and the last two terms in \eqref{eqinea1} are nonpositive and can be discarded. The left-hand side of \eqref{eqinea1} is then exactly $\mathcal{E}^{t+1}_{x^\star}$, and we compare what remains on the right-hand side with $\mathcal{E}^{t}_{x^\star}$ term by term; this is legitimate because all terms of $\mathcal{E}^{t}_{x^\star}$ are nonnegative, 
by \eqref{eqqnonneg}, and $\mathcal{E}^{0}_{x^\star}<+\infty$ because $u^0\in\mathrm{dom}\,\hc$. The ratio of the coefficients of $\sqn{z^t-x^\star}$ is $\frac{1}{1+a_t\gamma\mu_\g}$ and the ratio of the coefficients of $\mathcal{G}_{x^\star,u^\star}(x^t,u^t)$ is $\frac{a_{t+1}^2-a_{t+1}}{a_t^2}$. The dual terms require more care, because the coefficient of $-\sqn{\K^*(v^t-u^\star)}$ also changes from $\mathcal{E}^{t}_{x^\star}$ to the right-hand side of \eqref{eqinea1}. Let $\rho_t\coloneqq \frac{a_{t+1}^2}{a_t^2+a_t\tau\mu_\hc}$ be the ratio of the coefficients of $\sqn{v^t-u^\star}$; the condition $a_{t+1}^2\leq a_t^2+a_t\tau\mu_\hc$ means that $\rho_t\leq 1$. Since the function $a\mapsto \frac{a^2\gamma}{2+2a\gamma\mu_\g}$ is nondecreasing on $[0,+\infty)$ and $a_{t+1}\geq a_t$, we have $\frac{a_{t+1}^2\gamma}{2+2a_{t+1}\gamma\mu_\g}\geq \frac{a_{t}^2\gamma}{2+2a_{t}\gamma\mu_\g}\geq \rho_t\frac{a_{t}^2\gamma}{2+2a_{t}\gamma\mu_\g}$, so that
$\frac{a_{t+1}^2}{2\tau}\sqn{v^{t}-u^{\star}}  -\frac{a_{t+1}^2\gamma}{2+2a_{t+1}\gamma \mu_\g}\sqn{\K^*(v^{t}-u^{\star})}
\leq \rho_t\left(\frac{a_{t}^2+a_{t}\tau\mu_\hc}{2\tau}\sqn{v^{t}-u^{\star}}  -\frac{a_{t}^2\gamma}{2+2a_{t}\gamma \mu_\g}\sqn{\K^*(v^{t}-u^{\star})}\right)$.
This yields the claimed recursion. We note that when $\mu_\g=0$, the ratio of the coefficients of $\sqn{\K^*(v^t-u^\star)}$ is $a_{t+1}^2/a_t^2\geq\rho_t$, so that the condition $a_{t+1}^2\leq a_t^2+a_t\tau\mu_\hc$ is not needed for the recursion to hold; it is only needed when $\mu_\g>0$.

Under the additional condition $a_{t+1}^2-a_{t+1}\leq a_t^2$, the three terms in the maximum are at most $1$, hence $\mathcal{E}^{t}_{x^\star}\leq \mathcal{E}^0_{x^\star}$ and $\mathcal{G}_{x^\star,u^\star}(x^t,u^t)\leq \mathcal{E}^0_{x^\star}/a_t^2$. The particular sequence $(a_t)_{t\geq 0}$ given in the statement satisfies all the conditions, and its growth follows from $\sqrt{a_t^2+a_t\tau\mu_\hc}-a_t\geq \frac{a_t\tau\mu_\hc}{2a_t+\tau\mu_\hc}$ and $\frac{1+\sqrt{1+4a_t^2}}{2}-a_t\geq\frac{1}{2}$. Finally, $\sqn{u^t-u^\star}=\mathcal{O}(1/t^2)$ follows from $\mathcal{G}_{x^\star,u^\star}(x^t,u^t)\geq \frac{\mu_\hc}{2}\sqn{u^t-u^\star}$.
\end{proof}

\begin{corollary}[Accelerated linear convergence of \algn{APAPC} with $\mu_\g>0$ and $\mu_\hc>0$]\label{corcase1b}
Under the conditions of Theorem \ref{corcase1a}, suppose additionally that
$L_\ff\geq \mu_\g >0$. Then the solution $(x^\star,u^\star)$ of \eqref{eq2} is unique.
Let $a_\sharp \coloneqq \sqrt{\frac{L_\ff}{\mu_\g}}\geq 1$.
 In \algn{APAPC}, choose 
 $\gamma=\min\left(\frac{1}{L_\ff},\frac{1}{\|\K\|}\sqrt{\frac{\mu_\hc}{L_\ff}}\right)$, 
 $\tau=\frac{1}{\gamma\|\K\|^2}$, 
$a_0 \coloneqq \frac{\sqrt{\tau^2\mu^2_\hc+4}-\tau\mu_\hc}{2}$, $a_1=1$, and 
\begin{equation*}
a_{t+1}\coloneqq \min\left({\textstyle \sqrt{a_t^2+a_t\tau\mu_\hc},\frac{1+\sqrt{1+4a_t^2}}{2}},a_\sharp
\right),\quad\forall t\geq 1.
\end{equation*}
Let $T\geq 1$ be such that $a_T=a_\sharp$.
Then 
\algn{APAPC} achieves accelerated linear convergence:
\begin{equation*}
\mathcal{E}^{t+1}_{x^\star} \leq \max\left(\frac{1}{1+\sqrt{\frac{\mu_\g}{L_\ff}}},\frac{1}{1+\frac{\sqrt{\mu_\g \mu_\hc}}{\|\K\|}}\right)\mathcal{E}^t_{x^\star} ,\quad\forall t\geq T.
\end{equation*}
Consequently, 
its iteration complexity (to reach an accuracy $\mathcal{E}^t_{x^\star}
\leq \epsilon$) is
\begin{equation}\label{eqopr1}
\mathcal{O}\left(\left(\sqrt{\frac{L_\ff}{\mu_\g}}+\frac{\|\K\|}{\sqrt{\mu_\g\mu_{\hc}}}\right)\log \left(\frac{\mathcal{E}^0_{x^\star}}{\epsilon}\right)
\right).
\end{equation}
\end{corollary}
\begin{proof}
 We choose $\tau=\frac{1}{\gamma\|\K\|^2}$ and seek $a_\sharp\geq 1$ and $\gamma\leq \frac{1}{L_\ff}$ that minimize the maximum of the three terms: (i) $\frac{1}{1+a_\sharp \gamma\mu_\g}\geq \frac{1}{1+a_\sharp \frac{\mu_\g}{L_\ff}}$, (ii) $\frac{1}{1+\frac{\mu_\hc}{a_\sharp \gamma \|\K\|^2}}$, (iii) $1-\frac{1}{a_\sharp}\leq \frac{1}{1+\frac{1}{a_\sharp}}$. By choosing $a_\sharp = \sqrt{\frac{L_\ff}{\mu_\g}}$, (i) dominates (iii). Also, (i) dominates (ii) provided $\gamma\leq \frac{1}{\|\K\|}\sqrt{\frac{\mu_\hc}{L_\ff}}$. 
This leads to the stated choice of $\gamma$, and the contraction factor is
$\frac{1}{1+a_\sharp \gamma\mu_\g}= \max\Big(\left(1+\sqrt{\frac{\mu_\g}{L_\ff}}\right)^{-1},\left(1+\frac{\sqrt{\mu_\g \mu_\hc}}{\|\K\|}\right)^{-1}\Big)$.
\end{proof}

The complexity of \algn{APAPC} in \eqref{eqopr1} improves upon that of \algn{PAPC} ($a_t\equiv 1$), which, with appropriately chosen stepsizes, scales as $\frac{L_\ff}{\mu_\g}+\frac{\|\K\|}{\sqrt{\mu_\g\mu_{\hc}}}$: the first term, the primal dependence on $L_\ff/\mu_\g$, is not accelerated.

\begin{remark}
When $\h=0$, letting $L_\h\rightarrow 0$ (equivalently $\mu_\hc\rightarrow \infty$) in Theorem \ref{corcase1a} and Corollary \ref{corcase1b} correctly recovers the results established in Section \ref{secne} for minimizing $\ff+\g$ using \algn{APGD}.
\end{remark}
\begin{remark}\label{remhsmooth}
When $\h$ is smooth, Problem \eqref{eq1} consists of minimizing
 $\tilde{\ff}+\g$, where $\tilde{\ff}=\ff+\h\circ \K$ is $(L_\ff+\|\K\|^2L_\h)$-smooth. 
Such a problem can be solved directly by a primal method that evaluates the gradient $\nabla \tilde{\ff}(y^t)=\nabla \ff(y^t)+\K^* \nabla\h (\K y^t)$. As established in Corollary \ref{cor13}, \algn{APGD} applied to this formulation yields a complexity scaling as $\sqrt{\frac{L_\ff+\|\K\|^2L_\h}{\mu_\g}}$, matching \eqref{eqopr1}. Thus, in this regime, \algn{APAPC} does not improve upon the known complexity; neither do the other accelerated primal--dual algorithms analyzed in this regime \citep{dri24,he25,luo24}, which also handle a smooth $\h$ through the proximity operator of $\hc$. The interest of this regime lies elsewhere: it is one of the three situations of dual strong convexity covered by the same algorithm and the same Lyapunov analysis, and the only one in which a comparison with existing accelerated primal--dual methods is possible, so that \eqref{eqopr1} serves as a consistency check. The proximal treatment of $\h$ also has practical merits. \algn{APGD} must use a single stepsize, at most $1/(L_\ff+\|\K\|^2L_\h)$, for all components of the problem, which is restrictive when $L_\h\gg L_\ff$, as is the case for penalty terms or for smoothed nonsmooth functions with a small smoothing parameter. In \algn{APAPC}, the admissible stepsizes, $\gamma\leq 1/L_\ff$ and $\gamma\tau\|\K\|^2\leq 1$, do not depend on $L_\h$, which enters only through the dual strong convexity modulus $\mu_\hc=1/L_\h$ in the choice of the parameters driving acceleration, and $\h$ is handled exactly through its proximity operator. Finally, the same algorithm applies whether $\h$ is smooth or not; its smoothness is exploited in the analysis only.
\end{remark}

\subsection{Accelerated Convergence with $\K^*$ Bounded Below}\label{secinj}

In this section, we consider the case where $\K^*$ is bounded below; that is, we define 
\begin{equation}
\lambda_{\min}(\K\K^*) \eqdef \inf_{u\in\mathcal{U}, u \neq 0} \frac{\|\K^*u\|^2}{\|u\|^2},\label{eqlam1}
\end{equation}
and assume $\lambda_{\min}(\K\K^*) >0$. If $\mathcal{U}$ has finite dimension, $\lambda_{\min}(\K\K^*)$ is the smallest eigenvalue of $\K\K^*$. 
In this regime too, the dual problem \eqref{eq1cd} is strongly convex, ensuring that its solution $u^\star$ is unique.

\begin{theorem}[Accelerated convergence of \algn{APAPC} with $\lambda_{\min}(\K\K^*) >0$]\label{theoinj}
Suppose that $\lambda_{\min}(\K\K^*) >0$  and $\mu_\g\leq \frac{L_\ff}{2}$, and 
let $(x^\star,u^\star)$  be a solution of \eqref{eq2}.  In \algn{APAPC}, 
define the Lyapunov function 
\begin{align*}
\mathcal{E}^t_{x^\star}&\coloneqq \frac{2+a_{t}\gamma\mu_\g}{4\gamma}\sqn{z^{t}-x^\star}+
\frac{2a_{t}^2+\min\big(\frac{1}{2},\frac{1}{2a_{t}\gamma L_\ff}\big)
\gamma \tau a_{t}^2 \lambda_{\min}(\K\K^*)}{4\tau}
\sqn{v^{t}-u^{\star}} \\
&\quad -\frac{a_{t}^2\gamma}{2+2a_{t}\gamma \mu_\g}\sqn{\K^*(v^{t}-u^{\star})}+a_{t}^2 
\mathcal{G}_{x^\star,u^\star}(x^t,u^t),\quad\forall t\geq 0.
\end{align*}
Suppose that $u^0\in\mathrm{dom}\,\hc$, $\gamma\leq \frac{1}{2 L_\ff}$, $a_0>0$, $\gamma\tau\|\K\|^2\leq 1+a_0\gamma\mu_\g$, and, for every $t\geq 0$,
\begin{equation*}
a_t\leq a_{t+1}\leq \min\left(a_t\sqrt{1+\tfrac{1}{4}\gamma\tau\lambda_{\min}(\K\K^*)},\sqrt{a_t^2+\tfrac{\tau\lambda_{\min}(\K\K^*)}{4L_\ff}a_t}\right).
\end{equation*}
Then
\begin{equation*}
\mathcal{E}^{t+1}_{x^\star} \leq \max\left(\frac{1}{1+\frac{1}{2}a_{t}\gamma\mu_\g},\frac{a_{t+1}^2}{a_t^2+\frac{1}{4} a_{t}^2\gamma \tau \lambda_{\min}(\K\K^*)},\frac{a_{t+1}^2}{a_t^2+\frac{1}{4\gamma L_\ff} a_{t}\gamma\tau  \lambda_{\min}(\K\K^*)},
\frac{a_{t+1}^2-a_{t+1}}{a_t^2}\right)\mathcal{E}^t_{x^\star} ,
\end{equation*}
for every $t\geq 0$. Consequently, by choosing $\tau=\frac{\nu}{\gamma\|\K\|^2}$ for some $\nu\in (0,1]$, $a_0=a_1=1$, and
\begin{equation*}
a_{t+1}\coloneqq \min\left(a_t{\textstyle \sqrt{1+\frac{\nu\lambda_{\min}(\K\K^*)}{4 \|\K\|^2}}},{\textstyle \sqrt{a_t^2+a_{t}\frac{\nu\lambda_{\min}(\K\K^*)}{4\gamma L_\ff \|\K\|^2}},\frac{1+\sqrt{1+4a_t^2}}{2}}
\right),\quad\forall t\geq 1,
\end{equation*}
which grows as $a_t=\Omega\left( \frac{t}{2}\min\left(1,\frac{\nu\lambda_{\min}(\K\K^*)}{4\gamma L_\ff \|\K\|^2}\right)\right)$, we obtain
\begin{equation*}
\mathcal{E}^{t}_{x^\star} \leq \mathcal{E}^0_{x^\star} ,\quad\forall t\geq 0.
\end{equation*}
Therefore,
\begin{equation*}
\mathcal{G}_{x^\star,u^\star}(x^t,u^t)\leq \frac{\mathcal{E}^0_{x^\star} }{a_t^2} =\mathcal{O}\left(\frac{1}{t^2}\right),\quad\forall t\geq 0.
\end{equation*}
Moreover, if $\nu<1$, then $\sqn{v^{t}-u^{\star}}=\mathcal{O}(1/t^2)$; if $\nu=1$, then $\sqn{v^{t}-u^{\star}}=\mathcal{O}(1/t)$.
\end{theorem}

\begin{proof}
Let $(x^\star,u^\star)$  be a solution of \eqref{eq2} and  $t\geq 0$. Starting from \eqref{eqinea1}, which applies since $\gamma\tau\|\K\|^2\leq 1+a_0\gamma\mu_\g\leq 1+a_{t+1}\gamma\mu_\g$, we exploit the negative term  $-\frac{1}{2\gamma}\sqn{z^{t+1}-z^t }$ 
to produce a negative term $-\sqn{v^{t+1}-u^\star}$.
 To this end, we use Young's inequality $\|b+c+d\|^2 \geq \|b\|^2/2 -2\|c\|^2-2\|d\|^2$ for any vectors $b,c,d$. This yields
\begin{align*}
\frac{1}{2\gamma}\sqn{z^{t+1}-z^t} &=\frac{\gamma a_{t+1}^2 }{2}\sqn{\mu_\g (z^{t+1}-x^\star)+\nabla \ff(y^t)-\nabla \ff(x^\star)+\K^*(v^{t+1}-u^\star)}\notag\\
&\geq \frac{\gamma a_{t+1}^2}{4}\sqn{\K^*(v^{t+1}-u^\star)} -
a_{t+1}^2\gamma \mu^2_\g\sqn{z^{t+1}-x^\star} -a_{t+1}^2\gamma\sqn{\nabla \ff(y^t)-\nabla \ff(x^\star)}.
\end{align*}
Define $\delta_{t+1} \coloneqq  \min\big(\frac{1}{2},\frac{1}{2a_{t+1}\gamma L_\ff}\big)$. Then
\begin{align*}
-\frac{1}{2\gamma}\sqn{z^{t+1}-z^t}&=-\frac{1-\delta_{t+1}}{2\gamma}\sqn{z^{t+1}-z^t}-\frac{\delta_{t+1}}{2\gamma}\sqn{z^{t+1}-z^t}\\
&\leq -\frac{1}{4\gamma}\sqn{z^{t+1}-z^t}
-\frac{\delta_{t+1}\gamma a_{t+1}^2\lambda_{\min}(\K\K^*)}{4}\sqn{v^{t+1}-u^\star} +
\frac{a_{t+1} \mu^2_\g}{2L_\ff}\sqn{z^{t+1}-x^\star} \\
&\quad+\frac{a_{t+1}}{2L_\ff}
\sqn{\nabla \ff(y^t)-\nabla \ff(x^\star)},
\end{align*}
where we used $\delta_{t+1}\leq \frac{1}{2}$, 
$\delta_{t+1}\leq \frac{1}{2a_{t+1}\gamma L_\ff}$, 
and 
\begin{equation}
\sqn{\K^*(v^{t+1}-u^\star)}\geq \lambda_{\min}(\K\K^*) \sqn{v^{t+1}-u^\star}.\label{eqlowb}
\end{equation}
Substituting this bound into \eqref{eqinea1} gives
\begin{align*}
&\frac{1+a_{t+1}\gamma\mu_\g}{2\gamma}\sqn{z^{t+1}-x^\star}+\frac{a_{t+1}^2}{2\tau}\sqn{v^{t+1}-u^{\star}}  -\frac{a_{t+1}^2\gamma}{2+2a_{t+1}\gamma \mu_\g}\sqn{\K^*(v^{t+1}-u^{\star})}\notag\\
&\quad +a_{t+1}^2 \mathcal{G}_{x^\star,u^\star}(x^{t+1},u^{t+1})
\notag\\
&\leq\frac{1}{2\gamma}\sqn{z^{t}-x^\star}+\frac{a_{t+1}^2}{2\tau}\sqn{v^{t}-u^{\star}}  -\frac{a_{t+1}^2\gamma}{2+2a_{t+1}\gamma \mu_\g}\sqn{\K^*(v^{t}-u^{\star})}\\
&\quad -\left(\frac{1}{4\gamma}-\frac{L_\ff}{2}\right)\sqn{z^{t+1}-z^t } 
- \frac{\tau}{2}\sqn{ \K z^{t+1}-\tilde{\nabla}\hc(v^{t+1})}\notag\\
&\quad+(a_{t+1}^2-a_{t+1})\mathcal{G}_{x^\star,u^\star}(x^{t},u^{t}) -\frac{\delta_{t+1}\gamma a_{t+1}^2 \lambda_{\min}(\K\K^*)}{4}\sqn{v^{t+1}-u^\star} +
\frac{a_{t+1} \mu^2_\g}{2L_\ff}\sqn{z^{t+1}-x^\star} .\notag
\end{align*}
Using $\mu_\g\leq \frac{L_\ff}{2}$ and $\gamma\leq \frac{1}{2L_\ff}$, this simplifies to
\begin{align}
&\frac{2+a_{t+1}\gamma\mu_\g}{4\gamma}\sqn{z^{t+1}-x^\star}+\frac{2a_{t+1}^2+\delta_{t+1}\gamma \tau a_{t+1}^2 \lambda_{\min}(\K\K^*)}{4\tau}\sqn{v^{t+1}-u^{\star}} \notag\\
&\quad -\frac{a_{t+1}^2\gamma}{2+2a_{t+1}\gamma \mu_\g}\sqn{\K^*(v^{t+1}-u^{\star})}+a_{t+1}^2 \mathcal{G}_{x^\star,u^\star}(x^{t+1},u^{t+1})
\notag\\
&\leq\frac{2}{4\gamma}\sqn{z^{t}-x^\star}+\frac{2a_{t+1}^2}{4\tau}\sqn{v^{t}-u^{\star}}  -\frac{a_{t+1}^2\gamma}{2+2a_{t+1}\gamma \mu_\g}\sqn{\K^*(v^{t}-u^{\star})}\label{eqq27}\\
&\quad+(a_{t+1}^2-a_{t+1})\mathcal{G}_{x^\star,u^\star}(x^{t},u^{t})- \frac{\tau}{2}\sqn{ \K z^{t+1}-\tilde{\nabla}\hc(v^{t+1})}.\notag
\end{align}
This yields the claimed Lyapunov inequality. Indeed, proceeding as in the proof of Theorem~\ref{corcase1a}, we compare the right-hand side of \eqref{eqq27} with $\mathcal{E}^{t}_{x^\star}$ term by term, all terms of which are nonnegative, with $\mathcal{E}^{0}_{x^\star}<+\infty$ since $u^0\in\mathrm{dom}\,\hc$. The ratio for $\sqn{z^t-x^\star}$ is $\frac{2}{2+a_t\gamma\mu_\g}$ and the ratio for the Lagrangian gap is $\frac{a_{t+1}^2-a_{t+1}}{a_t^2}$. The ratio $\rho_t\coloneqq\frac{2a_{t+1}^2}{2a_t^2+\delta_t\gamma\tau a_t^2\lambda_{\min}(\K\K^*)}$ for $\sqn{v^t-u^\star}$ is the maximum of the second and third terms in the claimed bound, and the assumption on $a_{t+1}$ is precisely $\rho_t\leq 1$; since $a\mapsto \frac{a^2\gamma}{2+2a\gamma\mu_\g}$ is nondecreasing, the coefficient of $-\sqn{\K^*(v^t-u^\star)}$ on the right-hand side of \eqref{eqq27} is at least $\rho_t$ times its coefficient in $\mathcal{E}^{t}_{x^\star}$, which handles the dual terms. 
Finally, 
$\frac{a_t^2}{2\tau}(1-\nu)\sqn{v^{t}-u^{\star}} \leq \mathcal{E}^t_{x^\star}$, 
where $\nu=\gamma\tau\|\K\|^2\leq 1$. 
 If $\nu<1$, this implies $\sqn{v^{t}-u^{\star}}=\mathcal{O}(1/a_t^2)=\mathcal{O}(1/t^2)$. 
  If $\nu=1$, since $a_t\rightarrow \infty$, for sufficiently large $t$ we have
$\frac{ a_{t} \lambda_{\min}(\K\K^*)}{8  L_\ff}
\sqn{v^{t}-u^{\star}}\leq \mathcal{E}^t_{x^\star}$, which yields $\sqn{v^{t}-u^{\star}}=\mathcal{O}(1/t)$.
\end{proof}

\begin{corollary}[Accelerated linear convergence of \algn{APAPC} with $\mu_\g>0$ and  $\lambda_{\min}(\K\K^*) >0$]\label{corcase2b}
Under the conditions of Theorem \ref{theoinj}, suppose additionally that
$ \mu_\g >0$. Then the solution $(x^\star,u^\star)$ of \eqref{eq2} is unique.
In \algn{APAPC}, choose $\gamma\in\left[\frac{1}{2L_\ff}\sqrt{\frac{\lambda_{\min}(\K\K^*)}{\|\K\|^2}},\frac{1}{2L_\ff}\right]$, $\tau=\frac{1}{\gamma\|\K\|^2}$, and $a_0=a_1=1$. Define $a_\sharp\eqdef\max\left(\frac{1}{\gamma}\sqrt{\frac{\lambda_{\min}(\K\K^*)}{2\mu_\g L_\ff \|\K\|^2}},1\right)$ and choose
\begin{equation*}
a_{t+1}\coloneqq \min\left(a_t{\textstyle \sqrt{1+\frac{\lambda_{\min}(\K\K^*)}{4 \|\K\|^2}}},{\textstyle \sqrt{a_t^2+a_{t}\frac{\lambda_{\min}(\K\K^*)}{4\gamma L_\ff \|\K\|^2}},\frac{1+\sqrt{1+4a_t^2}}{2}},a_\sharp
\right),\quad\forall t\geq 1.
\end{equation*}
Let $T\geq 1$ be such that $a_T=a_\sharp$.
Then 
\algn{APAPC} achieves accelerated linear convergence:
\begin{equation*}
\mathcal{E}^{t+1}_{x^\star} \leq \max\left(
\frac{1}{1+\sqrt{\frac{\mu_\g \lambda_{\min}(\K\K^*)}{8 L_\ff \|\K\|^2}}},
\frac{1}{1+\frac{\lambda_{\min}(\K\K^*)}{4\|\K\|^2}}\right)
\mathcal{E}^t_{x^\star} ,\quad\forall t\geq T.
\end{equation*}
Consequently, 
its iteration complexity (to reach an accuracy $\mathcal{E}^t_{x^\star}
\leq \epsilon$) is
\begin{equation}\label{eqoprc2}
\mathcal{O}\left(\left(\sqrt{\frac{L_\ff \|\K\|^2}{\mu_\g \lambda_{\min}(\K\K^*)}}+\frac{\|\K\|^2}{\lambda_{\min}(\K\K^*)}\right)\log \left(\frac{\mathcal{E}^0_{x^\star}}{\epsilon}\right)
\right).
\end{equation}
\end{corollary}

\begin{proof}
We choose $\tau=\frac{1}{\gamma\|\K\|^2}$ and seek $a_\sharp\geq 1$ and $\gamma\leq \frac{1}{2L_\ff}$ that minimize the maximum of the four terms: (i) $\left(1+\frac{1}{2}a_\sharp \gamma\mu_\g\right)^{-1}$, (ii) $\left(1+\frac{ \lambda_{\min}(\K\K^*)}{4\gamma L_\ff a_\sharp \|\K\|^2} \right)^{-1}$, (iii) $1-\frac{1}{a_\sharp}\leq \left(1+\frac{1}{a_\sharp}\right)^{-1}$, (iv) $\left(1+\frac{\lambda_{\min}(\K\K^*)}{4\|\K\|^2}\right)^{-1}$.  First, we rewrite (ii) as 
$\left(1+\frac{ \lambda_{\min}(\K\K^*)}{4\gamma L_\ff a_\sharp \|\K\|^2} \right)^{-1}=\max\!\left(\!\left(1+\sqrt{\frac{\mu_\g\lambda_{\min}(\K\K^*)}{8L_\ff\|\K\|^2}}\right)^{-1},
\left(1+\frac{ \lambda_{\min}(\K\K^*)}{4\gamma L_\ff  \|\K\|^2} \right)^{-1}\right)$. Since $\gamma\le \frac{1}{2L_\ff}$, the second term is dominated by (iv). Moreover, 
since $\frac{1}{a_\sharp}= \min\left(\gamma\sqrt{\frac{2\mu_\g L_\ff \|\K\|^2}{\lambda_{\min}(\K\K^*)}},1\right)\ge \sqrt{\frac{\mu_\g  }{2 L_\ff}}$, 
 (ii) dominates (iii). 
Finally, for  (i), since  $a_\sharp\gamma\ge \sqrt{\frac{\lambda_{\min}(\K\K^*)}{2\mu_\g L_\ff \|\K\|^2}}$, it follows that  $\left(1+\frac{1}{2}a_\sharp \gamma\mu_\g\right)^{-1}\leq \left(1+\frac{1}{2}\sqrt{\frac{\mu_\g\lambda_{\min}(\K\K^*)}{2 L_\ff \|\K\|^2}}\right)^{-1}$, which is dominated by (ii).  This yields the stated contraction factor.
\end{proof}

The complexity of \algn{APAPC} in \eqref{eqoprc2} improves upon that of \algn{PAPC}, which, with appropriately chosen stepsizes, scales as $\frac{L_\ff}{\mu_\g}+\frac{\|\K\|^2}{\lambda_{\min}(\K\K^*)}$.

\begin{remark}\label{remkid}
Consider the case where $\K=\Id$ and $\mu_\g=0$, which implies $\|\K\|=\lambda_{\min}(\K\K^*)=1$. By setting $\tau=1/\gamma$ in \algn{APAPC} and invoking the Moreau identity, the dual variable $v^t$ cancels out and the primal update collapses to
\begin{align*}
z^{t+1}&=z^t-a_{t+1}\gamma\nabla \ff(y^t)-a_{t+1}\gamma \mathrm{prox}_{\frac{1}{a_{t+1}\gamma} \hc} \left(v^t+{\textstyle\frac{1}{a_{t+1}\gamma}} \big(z^t-a_{t+1}\gamma\nabla \ff(y^t)-a_{t+1}\gamma v^t\big)\right)\\
&=\mathrm{prox}_{a_{t+1}\gamma \h}\left(z^t-a_{t+1}\gamma\nabla \ff(y^t)\right).
\end{align*}
Thus, \algn{APAPC} reduces exactly to \algn{APGD}, with $\h$ playing the role of $\g$. Consequently, Theorem \ref{theoinj} consistently recovers the accelerated $\mathcal{O}(1/t^2)$ convergence rate established in Section \ref{secne} for minimizing $\ff+\h$ via \algn{APGD}.
\end{remark}

\subsection{Accelerated Convergence for Linearly Constrained Problems}\label{seclc}

Let $b\in \mathrm{ran}(\K)$, 
the range of $\K$. In this section, we consider the linearly constrained problem
\begin{equation}
\minimize_{x\in\mathcal{X}} \; \ff(x)+\frac{\mu_\g}{2}\|x\|^2\quad\mbox{s.t.}\ \K x=b.\label{eq1lc}
\end{equation}
This is a special case of Problem \eqref{eq1c} with $\h(u)=(0$ if $u=b$, $+\infty$ otherwise$)$. Its conjugate is $\hc(u)=\langle u,b\rangle$.
The associated dual problem is therefore
\begin{equation}
\minimize_{u\in\mathcal{U}} \; \left(\ff+\frac{\mu_\g}{2}\|\cdot\|^2\right)^*(-\K^*u)+\langle u,b\rangle. \label{eq1dlc}
\end{equation}
We define $\lambda_{\min}^+(\K\K^*)$ as the infimum of the spectrum of $\K\K^*$ excluding zero; that is,
\begin{equation*}
\lambda^+_{\min}(\K\K^*) \coloneqq \inf_{u\in\overline{\mathrm{ran}(\K)}, u \neq 0} \frac{\|\K^*u\|^2}{\|u\|^2},
\end{equation*}
and we assume $\lambda^+_{\min}(\K\K^*) >0$ (this implies that $\mathrm{ran}(\K)$ is closed). This condition is only a safeguard for the infinite-dimensional case: if $\mathcal{U}$ has finite dimension, $\lambda^+_{\min}(\K\K^*)$ is the smallest nonzero eigenvalue of $\K\K^*$, so no additional assumption on $\K$ is required.
Note that $\lambda_{\min}(\K\K^*)$, defined in \eqref{eqlam1}, may still be zero. 

In \algn{APAPC}, we assume that $u^0=v^0 \in\mathrm{ran}(\K)$. Step \ref{eqapapcc2k} simplifies to
\begin{equation*}
v^{t+1}\coloneqq v^t+{\textstyle\frac{\tau}{a_{t+1}}} (\K \hat{z}^t-b).
\end{equation*}
Therefore, by induction, $u^t,v^t \in\mathrm{ran}(\K)$ for every $t\geq 0$. 
This key property allows us to reuse the analysis of Section \ref{secinj}, with the improvement that $\lambda_{\min}(\K\K^*)$ can be replaced by $\lambda^+_{\min}(\K\K^*)$ in the crucial inequality \eqref{eqlowb}.
Equivalently, the algorithm can be viewed as operating on the real Hilbert space $\widetilde{\mathcal{U}}\coloneqq \mathrm{ran}(\K)$; redefining $\K$ as a mapping from $\mathcal{X}$ to $\widetilde{\mathcal{U}}$ leads precisely to the definition of $\lambda^+_{\min}(\K\K^*)$. This property is specific to the linearly constrained setting, where $\mathrm{prox}_{\hc}$ maps $\mathrm{ran}(\K)$ into itself.

Although the dual problem
 \eqref{eq1dlc} is not necessarily strongly convex on $\mathcal{U}$, its restriction to $\mathrm{ran}(\K)$ is strongly convex. We therefore define $u^\star$ as the unique solution of \eqref{eq1dlc} in $\mathrm{ran}(\K)$. The optimality conditions are
\begin{equation}
\left\{\begin{array}{l}
0 \in  \nabla \ff(x^\star)+\mu_\g x^\star+\K^*u^\star,\\
u^\star\in \mathrm{ran}(\K),\\
\K x^\star=b,\end{array}\right.\label{eq2lc}
\end{equation}
and we assume that a primal solution $x^\star$ exists. 
The Lagrangian gap satisfies $\mathcal{G}_{x^\star,u^\star}(x,u)=(\ff+\g)(x)-(\ff+\g)(x^\star)+\langle \K x-b,u^\star\rangle=(\ff+\g)(x)-(\ff+\g)(x^\star)-\langle x-x^\star,\nabla \ff (x^\star) + \mu_\g x^\star\rangle$ for every $x\in\mathcal{X}$ and $u\in \mathrm{ran}(\K)$.
Since this quantity is independent of $u$ and $x^\star$, we denote it by $\mathcal{G}(x)$ in this section. Consequently, the analysis involves only the variable $v^t$, and $u^t$ can be ignored. 
We now state the main result.

\begin{theorem}[Accelerated convergence of \algn{APAPC} for solving \eqref{eq2lc} with $\lambda^+_{\min}(\K\K^*) >0$]\label{theolc}
Suppose that $\lambda^+_{\min}(\K\K^*) >0$  and $\mu_\g\leq \frac{L_\ff}{2}$, and 
let $(x^\star,u^\star)$  be a solution of \eqref{eq2lc}.  In \algn{APAPC}, 
define the Lyapunov function 
\begin{align*}
\mathcal{E}^t_{x^\star}&\coloneqq \frac{2+a_{t}\gamma\mu_\g}{4\gamma}\sqn{z^{t}-x^\star}+
\frac{2a_{t}^2+\min\big(\frac{1}{2},\frac{1}{2a_{t}\gamma L_\ff}\big)
\gamma \tau a_{t}^2 \lambda^+_{\min}(\K\K^*)}{4\tau}
\sqn{v^{t}-u^{\star}}  \\
&\quad -\frac{a_{t}^2\gamma}{2+2a_{t}\gamma \mu_\g}\sqn{\K^*(v^{t}-u^{\star})}+a_{t}^2 
\mathcal{G}(x^t)+\frac{a_{t}\tau}{2}\sqn{\K x^{t}-b},\quad\forall t\geq 0.
\end{align*}
Suppose that $u^0=v^0\in\mathrm{ran}(\K)$, $\gamma\leq \frac{1}{2 L_\ff}$, $a_0>0$, $\gamma\tau\|\K\|^2\leq 1+a_0\gamma\mu_\g$, and, for every $t\geq 0$,
\begin{equation*}
a_t\leq a_{t+1}\leq \min\left(a_t\sqrt{1+\tfrac{1}{4}\gamma\tau\lambda^+_{\min}(\K\K^*)},\sqrt{a_t^2+\tfrac{\tau\lambda^+_{\min}(\K\K^*)}{4L_\ff}a_t}\right).
\end{equation*}
Then 
\begin{equation*}
\mathcal{E}^{t+1}_{x^\star} \leq \max\left(\frac{1}{1+\frac{1}{2}a_{t}\gamma\mu_\g},\frac{a_{t+1}^2}{a_t^2+\frac{1}{4} a_{t}^2\gamma \tau \lambda^+_{\min}(\K\K^*)},\frac{a_{t+1}^2}{a_t^2+\frac{1}{4\gamma L_\ff} a_{t}\gamma\tau  \lambda^+_{\min}(\K\K^*)},
\frac{a_{t+1}^2-a_{t+1}}{a_t^2}\right)\mathcal{E}^t_{x^\star},
\end{equation*}
for every $t\geq 0$. Consequently, by choosing $\tau=\frac{\nu}{\gamma\|\K\|^2}$ for some $\nu\in (0,1]$, $a_0=a_1=1$, and 
\begin{equation*}
a_{t+1}\coloneqq \min\left(a_t{\textstyle \sqrt{1+\frac{\nu\lambda^+_{\min}(\K\K^*)}{4 \|\K\|^2}}},{\textstyle \sqrt{a_t^2+a_{t}\frac{\nu\lambda^+_{\min}(\K\K^*)}{4\gamma L_\ff \|\K\|^2}},\frac{1+\sqrt{1+4a_t^2}}{2}}
\right),\quad\forall t\geq 1,
\end{equation*}
which grows as $a_t=\Omega\left( \frac{t}{2}\min\left(1,\frac{\nu\lambda^+_{\min}(\K\K^*)}{4\gamma L_\ff \|\K\|^2}\right)\right)$, we obtain
\begin{equation*}
\mathcal{E}^{t}_{x^\star} \leq \mathcal{E}^0_{x^\star} ,\quad\forall t\geq 0.
\end{equation*}
Therefore, for every $t\geq 0$,
\begin{equation*}
\mathcal{G}(x^t)\leq \frac{\mathcal{E}^0_{x^\star} }{a_t^2} =\mathcal{O}\left(\frac{1}{t^2}\right)
,\quad \sqn{\K x^{t}-b}\leq \frac{2\mathcal{E}^0_{x^\star} }{a_t \tau}= 
\mathcal{O}\left(\frac{1}{t}\right).
\end{equation*}
Moreover, if $\nu<1$, then $\sqn{v^{t}-u^{\star}}=\mathcal{O}(1/t^2)$; if $\nu=1$, then $\sqn{v^{t}-u^{\star}}=\mathcal{O}(1/t)$.
\end{theorem}

\begin{proof}
The proof follows that of Theorem~\ref{theoinj}, with $\lambda_{\min}(\K\K^*)$ replaced by $\lambda^+_{\min}(\K\K^*)$. 
 Let $(x^\star,u^\star)$  be a solution of \eqref{eq2lc} and  $t\geq 0$. Starting from \eqref{eqq27} with $\tilde{\nabla}\hc(v^{t+1})=b$, we obtain 
 \begin{align}
&\frac{2+a_{t+1}\gamma\mu_\g}{4\gamma}\sqn{z^{t+1}-x^\star}+\frac{2a_{t+1}^2+\delta_{t+1}\gamma \tau a_{t+1}^2 \lambda^+_{\min}(\K\K^*)}{4\tau}\sqn{v^{t+1}-u^{\star}} \notag\\
&\quad -\frac{a_{t+1}^2\gamma}{2+2a_{t+1}\gamma \mu_\g}\sqn{\K^*(v^{t+1}-u^{\star})}+a_{t+1}^2 \mathcal{G}(x^{t+1})
\notag\\
&\leq\frac{2}{4\gamma}\sqn{z^{t}-x^\star}+\frac{2a_{t+1}^2}{4\tau}\sqn{v^{t}-u^{\star}}  -\frac{a_{t+1}^2\gamma}{2+2a_{t+1}\gamma \mu_\g}\sqn{\K^*(v^{t}-u^{\star})}\notag\\
&\quad+(a_{t+1}^2-a_{t+1})\mathcal{G}(x^{t})- \frac{\tau}{2}\sqn{ \K z^{t+1}-b}.\label{eq99}
\end{align}
By Jensen's inequality applied to Step~\ref{eqapapcc4k}, we have
$\sqn{\K x^{t+1}-b}\leq \Big(1-\frac{1}{a_{t+1}}\Big) \sqn{\K x^{t}-b}+\frac{1}{a_{t+1}} \sqn{\K z^{t+1}-b}$; that is, 
$\sqn{\K z^{t+1}-b}\geq a_{t+1}\sqn{\K x^{t+1}-b}-(a_{t+1}-1)\sqn{\K x^{t}-b}$.
Hence, 
\begin{align*}
&\frac{2+a_{t+1}\gamma\mu_\g}{4\gamma}\sqn{z^{t+1}-x^\star}+\frac{2a_{t+1}^2+\delta_{t+1}\gamma \tau a_{t+1}^2 \lambda^+_{\min}(\K\K^*)}{4\tau}\sqn{v^{t+1}-u^{\star}} \notag\\
&\quad -\frac{a_{t+1}^2\gamma}{2+2a_{t+1}\gamma \mu_\g}\sqn{\K^*(v^{t+1}-u^{\star})}+a_{t+1}^2 \mathcal{G}(x^{t+1})+\frac{a_{t+1}\tau}{2}\sqn{\K x^{t+1}-b} 
\notag\\
&\leq\frac{2}{4\gamma}\sqn{z^{t}-x^\star}+\frac{2a_{t+1}^2}{4\tau}\sqn{v^{t}-u^{\star}}  -\frac{a_{t+1}^2\gamma}{2+2a_{t+1}\gamma \mu_\g}\sqn{\K^*(v^{t}-u^{\star})}\\
&\quad+(a_{t+1}^2-a_{t+1})\mathcal{G}(x^{t})+\frac{(a_{t+1}-1)\tau}{2}\sqn{\K x^{t}-b}.
\end{align*}
Moreover, $\frac{a_{t+1}-1}{a_t}=\frac{a_{t+1}^2-a_{t+1}}{a_t a_{t+1}}\leq \frac{a_{t+1}^2-a_{t+1}}{a_t^2}$. This yields the claimed recursion, the dual terms being handled exactly as in the proof of Theorem~\ref{theoinj}.
\end{proof}

With a refined analysis, we improve the feasibility rate from $\mathcal{O}(1/\sqrt{t})$ to $\mathcal{O}(1/t)$:
\begin{proposition}[Feasibility gap and objective gap]\label{proplcfeas}
Under the conditions of Theorem~\ref{theolc}, with the parameters chosen as in its second part, suppose that $\mu_\g=0$ and $\nu<1$. Then, for every $t\geq 1$,
\begin{equation*}
\|\K x^{t}-b\|\leq \frac{1}{a_t^2}\left(\|\K x^{0}-b\|+\sqrt{\frac{2\mathcal{E}^0_{x^\star}}{\tau(1-\nu)}}\,\big(a_t+2t\big)\right)=\mathcal{O}\left(\frac{1}{t}\right)
\end{equation*}
and $|\ff(x^t)-\ff(x^\star)|\leq \mathcal{G}(x^t)+\|u^\star\|\,\|\K x^{t}-b\|=\mathcal{O}(1/t)$.
\end{proposition}

\begin{proof}
Let $t\geq 0$. Since $\mu_\g=0$ and $\tilde{\nabla}\hc(v^{t+1})=b$, 
the identity \eqref{eqqvv} reads
 $a_{t+1}^2 \mathsf{M} (v^{t+1}-v^t)=a_{t+1}(\K z^{t+1}-b)$, where $\mathsf{M}\coloneqq \frac{1}{\tau}\Id-\gamma\K\K^*$ is self-adjoint and satisfies $\frac{1-\nu}{\tau}\Id\preceq \mathsf{M}\preceq \frac{1}{\tau}\Id$, since $\gamma\tau\|\K\|^2=\nu$. Moreover, Step~\ref{eqapapcc4k} gives $\K z^{t+1}-b=a_{t+1}(\K x^{t+1}-b)-(a_{t+1}-1)(\K x^{t}-b)$. Setting $r^t\coloneqq a_t^2(\K x^t-b)$ and $\rho_t\coloneqq (a_{t+1}^2-a_{t+1})/a_t^2$, which belongs to $[0,1]$ since $1\leq a_{t+1}\leq \frac{1+\sqrt{1+4a_t^2}}{2}$, we obtain $r^{t+1}=\rho_t r^t+a_{t+1}^2\,\mathsf{M}(v^{t+1}-v^t)$. 
Let $T\geq 1$. For $t=0,\ldots,T-1$, define $P_t\coloneqq \prod_{s=t+1}^{T-1}\rho_s\in[0,1]$ (with $P_{T-1}=1$) and $c_t\coloneqq P_t a_{t+1}^2$. Unrolling the recursion yields
\begin{equation*}
r^T=\Big(\prod_{s=0}^{T-1}\rho_s\Big) r^0+\sum_{t=0}^{T-1}c_t\,\mathsf{M}(v^{t+1}-v^t).
\end{equation*}
The key observation is that the weights $c_t$ are nondecreasing, with controlled increments: for $t=1,\ldots,T-1$, since $P_{t-1}=P_t\rho_t$,
$c_t-c_{t-1}=P_t\big(a_{t+1}^2-\rho_t a_t^2\big)=P_t\, a_{t+1}\in[0,a_{t+1}]$.
Hence, by Abel summation, using $c_{T-1}=a_T^2$ and $c_0\leq a_1^2=1$,
\begin{equation*}
\sum_{t=0}^{T-1}c_t(v^{t+1}-v^t)=a_T^2(v^T-u^\star)-c_0(v^0-u^\star)-\sum_{t=1}^{T-1}(c_t-c_{t-1})(v^t-u^\star).
\end{equation*}
On the other hand, in $\mathcal{E}^t_{x^\star}$, the terms other than the two involving $v^t-u^\star$ are nonnegative, and $\gamma\|\K^*(v^t-u^\star)\|^2\leq \gamma\|\K\|^2\|v^t-u^\star\|^2=\frac{\nu}{\tau}\|v^t-u^\star\|^2$, so that $\mathcal{E}^t_{x^\star}\geq \frac{(1-\nu)a_t^2}{2\tau}\sqn{v^t-u^\star}$. Combined with $\mathcal{E}^t_{x^\star}\leq \mathcal{E}^0_{x^\star}$, this gives $\|v^t-u^\star\|\leq \chi/a_t$ for every $t\geq 0$, where $\chi\coloneqq\sqrt{2\tau\mathcal{E}^0_{x^\star}/(1-\nu)}$. 
Finally, $a_{t+1}\leq 2a_t$ for every $t\geq 0$, by \eqref{eqagrow} and $a_0=1$.
 Putting everything together and using $\|\mathsf{M}\|\leq 1/\tau$ and $r^0=\K x^0-b$, we get
\begin{equation*}
\|r^T\|\leq \|\K x^0-b\|+\frac{1}{\tau}\Big(a_T^2\frac{\chi}{a_T}+\frac{\chi}{a_0}+\sum_{t=1}^{T-1}a_{t+1}\frac{\chi}{a_t}\Big)\leq \|\K x^0-b\|+\frac{\chi}{\tau}\big(a_T+1+2(T-1)\big),
\end{equation*}
which is the first claim, since $\chi/\tau=\sqrt{2\mathcal{E}^0_{x^\star}/(\tau(1-\nu))}$; the $\mathcal{O}(1/T)$ rate follows from the linear growth of $a_T$ established in Theorem~\ref{theolc}.
For the second claim, we have $\ff(x^T)-\ff(x^\star)=\mathcal{G}(x^T)-\langle \K x^T-b,u^\star\rangle$ with $\mathcal{G}(x^T)\geq 0$, so that $-\|u^\star\|\,\|\K x^T-b\|\leq \ff(x^T)-\ff(x^\star)\leq \mathcal{G}(x^T)+\|u^\star\|\,\|\K x^T-b\|$, and the rate follows from Theorem~\ref{theolc} and the first claim.
\end{proof}

\begin{corollary}[Accelerated linear convergence of \algn{APAPC} on \eqref{eq2lc} with $\mu_\g>0$ and  $\lambda^+_{\min}(\K\K^*) >0$]\label{corcase3b}
Under the conditions of Theorem \ref{theolc}, suppose additionally that
$ \mu_\g >0$. Then the solution $(x^\star,u^\star)$ of \eqref{eq2lc} is unique. 
 In \algn{APAPC}, 
choose $\gamma\in\left[\frac{1}{2L_\ff}\sqrt{\frac{\lambda^+_{\min}(\K\K^*)}{\|\K\|^2}},\frac{1}{2L_\ff}\right]$,  $\tau=\frac{1}{\gamma\|\K\|^2}$, and $a_0=a_1=1$.  Define
$a_\sharp\eqdef\max\left(\frac{1}{\gamma}\sqrt{\frac{\lambda^+_{\min}(\K\K^*)}{2\mu_\g L_\ff \|\K\|^2}},1\right)$ 
  and choose
\begin{equation*}
a_{t+1}\coloneqq \min\left(a_t{\textstyle \sqrt{1+\frac{\lambda^+_{\min}(\K\K^*)}{4 \|\K\|^2}}},{\textstyle \sqrt{a_t^2+a_{t}\frac{\lambda^+_{\min}(\K\K^*)}{4\gamma L_\ff \|\K\|^2}},\frac{1+\sqrt{1+4a_t^2}}{2}},a_\sharp
\right),\quad\forall t\geq 1.
\end{equation*}
Let $T\geq 1$ be such that $a_T=a_\sharp$. Then 
\algn{APAPC} achieves accelerated linear convergence:
\begin{equation*}
\mathcal{E}^{t+1}_{x^\star} \leq \max\left(
\left(1+\sqrt{\frac{\mu_\g \lambda^+_{\min}(\K\K^*)}{ 8L_\ff \|\K\|^2}}\right)^{-1},
\left(1+\frac{\lambda^+_{\min}(\K\K^*)}{4\|\K\|^2}\right)^{-1}\right)
\mathcal{E}^t_{x^\star} ,\quad\forall t\geq T.
\end{equation*}
Consequently, 
its iteration complexity (to reach an accuracy $\mathcal{E}^t_{x^\star}
\leq \epsilon$) is
\begin{equation}\label{eqoprc3}
\mathcal{O}\left(\left(\sqrt{\frac{L_\ff \|\K\|^2}{\mu_\g \lambda^+_{\min}(\K\K^*)}}+\frac{\|\K\|^2}{\lambda^+_{\min}(\K\K^*)}\right)\log \left(\frac{\mathcal{E}^0_{x^\star}}{\epsilon}\right)
\right).
\end{equation}
\end{corollary}

\begin{remark}
Chebyshev acceleration can further improve the complexity factor in \eqref{eqoprc3} to $\sqrt{\frac{L_\ff \|\K\|^2}{\mu_\g \lambda^+_{\min}(\K\K^*)}}$, which is optimal \citep{sal22}. This amounts to replacing the constraint $\K x=b$ in \eqref{eq1lc} by $\sqrt{P(\K^*\K)}x=\sqrt{P(\K^*\K)}x^\star$ for a suitably chosen polynomial $P$, effectively applying $\K^*\K$ multiple times per iteration.
\end{remark}

\subsection{Point Convergence of \algn{APAPC}}\label{secpointa}

In regimes where both the primal and dual problems are strongly convex, the linear convergence of \algn{APAPC} guarantees the strong convergence of the iterates to the unique saddle-point solution of \eqref{eq2}; that is, $\|u^t-u^\star\|\rightarrow 0$ and $\|x^t-x^\star\|\rightarrow 0$. In this section, we therefore focus on the sublinear regime where $\mu_\g=0$ and only the dual problem is assumed to be strongly convex. 
Under these conditions, our preceding analysis guarantees that the dual iterates $v^t$ converge strongly to the unique dual solution $u^\star$; the following remark transfers this property to $u^t$, and we then establish the weak convergence of the primal iterates.

\begin{remark}\label{rempolyak}
Step~\ref{eqapapcc5k} of \algn{APAPC} has the averaging form of Lemma~\ref{lemavg}, with $u^0=v^0$, so that $u^t=\sum_{k=0}^t w_{k,t}v^k$ for every $t\geq 0$, with the weights defined in \eqref{eqavg}. Consequently, for any dual solution $u^\star$, $\|u^t-u^\star\|\leq \sum_{k=0}^t w_{k,t}\|v^k-u^\star\|$, and if $\|v^{t} - u^\star\|\rightarrow 0$ and $\sum_{t\geq 1} \frac{1}{a_t} = \infty $, then $\|u^t - u^\star\|\rightarrow 0$ as well, by Lemma~\ref{lemavg}$\mathrm{(ii)}$.
\end{remark}

\begin{theorem}[Point Convergence of \algn{APAPC}]\label{theoconva}
Suppose that $\mu_\g=0$ and that the following assumptions hold in \algn{APAPC}:

$\mathrm{(i)}$ 
$a_1=1$, $a_t \rightarrow \infty$ and  $\sum_{t\geq 1} \frac{1}{a_t}=\infty$.

$\mathrm{(ii)}$ The solution $u^\star$ of the dual problem \eqref{eq1cd} is unique, and $\|v^t-u^\star\|\rightarrow 0$. 

$\mathrm{(iii)}$ The sequence $(z^t)_{t \ge 0}$ is bounded.

$\mathrm{(iv)}$ $\mathcal{G}_{x^\star,u^\star}(x^t,u^t)\rightarrow 0$ for every saddle-point solution $(x^\star,u^\star)$ to \eqref{eq2}.

$\mathrm{(v)}$ $\big\|\K z^t - \tilde{\nabla}\hc (v^t)\big\|\rightarrow 0$.

\noindent Then, the following hold:

$\mathrm{(a)}$  $\|u^t-u^\star\|\rightarrow 0$ and $\|y^t-x^t\|=\mathcal{O}(1/a_{t+1})$.

$\mathrm{(b)}$ For every $t\geq 0$, $\sqn{\nabla \ff(x^t)+\K^* u^\star}\le 2L_\ff\mathcal{G}_{x^\star,u^\star}(x^t,u^t)\to 0$.

$\mathrm{(c)}$ The sequence $(x^t)_{t \ge 0}$ is bounded, and all its weak cluster points are solutions to \eqref{eq1c}. 

$\mathrm{(d)}$ Every weak cluster point $z^\infty$ of $(z^t)_{t \ge 0}$ satisfies $\K z^\infty \in \partial \hc(u^\star)$.

\noindent Suppose, in addition, that the following hold:

$\mathrm{(vi)}$ For every saddle-point solution $(x^\star,u^\star)$ to \eqref{eq2}, there exists a Lyapunov function 

$\mathcal{E}^t_{x^\star}\coloneqq \frac{1}{2\gamma}\sqn{z^{t}-x^\star}+a_t^2\mathcal{G}_{x^\star,u^\star}(x^t,u^t)+\mathcal{R}^t$ for some dual term $\mathcal{R}^t\geq 0$ that does not depend on 
$x^\star$. Moreover, $(\mathcal{E}^t_{x^\star})_{t\ge 0}$ is nonincreasing.

$\mathrm{(vii)}$ For any pair of solutions $(\hat{x}^\star,u^\star)$ and $(\check{x}^\star,u^\star)$ to \eqref{eq2}, we have $\K\hat{x}^\star=\K \check{x}^\star$.

\noindent Then, there exists a solution $x^\star$ to \eqref{eq1c} such that $(x^t)_{t\geq 0}$ and $(y^t)_{t\geq 0}$ both converge weakly to $x^\star$. 
\end{theorem}

\begin{proof}
Assumptions $\mathrm{(i)}$ and $\mathrm{(ii)}$ imply that $\|u^t-u^\star\|\rightarrow 0$, as established in Remark \ref{rempolyak}. 
Step~\ref{eqapapcc4k} of \algn{APAPC} has the averaging form of Lemma~\ref{lemavg}, with $x^0=z^0$ and $\sum_{t\geq 1}\frac{1}{a_t}=\infty$ by Assumption $\mathrm{(i)}$; since $(z^t)_{t\geq 0}$ is assumed to be bounded, $(x^t)_{t\geq 0}$ is bounded by Lemma~\ref{lemavg}$\mathrm{(i)}$, and therefore has at least one weak cluster point.
 Furthermore, from the update in Step~\ref{eqapapcs1k} of \algn{APAPC}, we have $\|y^t-x^t\|=\frac{1}{a_{t+1}}\|x^t-z^t\|$. Since $(x^t-z^t)_{t\geq 0}$ is bounded, this yields $\|y^t-x^t\|=\mathcal{O}(1/a_{t+1})$. 

Let $z^\infty$ be a weak cluster point of $(z^t)_{t \ge 0}$, meaning there exists a subsequence $(z^{t_j})_{j \ge 0}$ such that $z^{t_j} \rightharpoonup z^\infty$, where $\rightharpoonup$ denotes weak convergence. Because $\K$ is bounded, it is weakly continuous, which implies $\K z^{t_j} \rightharpoonup \K z^\infty$. 
By Assumption $\mathrm{(v)}$, we have $\big\|\K z^{t_j} - \tilde{\nabla}\hc(v^{t_j})\big\| \to 0$. 
Therefore, $\tilde{\nabla}\hc(v^{t_j}) \rightharpoonup \K z^\infty$. Moreover, 
by Assumption $\mathrm{(ii)}$, we know that $\|v^{t_j} - u^\star\|\to 0$.
The subdifferential $\partial \hc$ is a maximally monotone operator, meaning its graph is strongly-weakly sequentially closed \citep[Proposition 20.38]{bau17}. 
It follows that the limit point $(u^\star, \K z^\infty)$ belongs to $\mathrm{gra}(\partial \hc)$,  where $\mathrm{gra}$ denotes the graph; that is, 
$\K z^\infty \in \partial \hc(u^\star)$.

Now, let us characterize the weak cluster points of $(x^t)_{t\geq 0}$.

\textbf{Step 1: The weak cluster points of $(x^t)_{t\geq 0}$ are primal solutions.} 
Given any solution $(x^\star,u^\star)$ to \eqref{eq2}, which therefore satisfies $\nabla \ff(x^\star)=-\K^* u^\star$, 
we define the Bregman distance $D_\ff: x\in\mathcal{X} \mapsto  \ff(x)-\ff(x^\star)+\langle x-x^\star,\K^* u^\star\rangle$. 
For every $t\geq 0$, 
we have $\mathcal{G}_{x^\star,u^\star}(x^t,u^t) \ge D_\ff(x^t) \ge \frac{1}{2L_\ff}\sqn{\nabla \ff(x^t)+\K^* u^\star}$ \citep[Lemma 3.4]{bub15}. 
Therefore, by Assumption $\mathrm{(iv)}$, $\big\|\nabla \ff(x^t)+\K^* u^\star\big\|\to 0$. The gradient $\nabla \ff$ is a maximally monotone operator, meaning its graph  is weakly-strongly sequentially closed \citep[Proposition 20.38]{bau17}. 
Consequently, for every weak cluster point $x^\infty$ of  $(x^t)_{t\geq 0}$, the weak convergence of the subsequence combined with the strong convergence of the gradient implies that $\nabla \ff(x^\infty) = -\K^* u^\star$. This fulfills the first optimality condition in \eqref{eq2}.

Let $x^\infty$ be a weak cluster point of $(x^t)_{t\geq 0}$. We now prove the second optimality condition, $\K x^\infty \in \partial \hc(u^\star)$. 
Because $\partial \hc$ is a maximally monotone operator, this inclusion is equivalent to showing that for every pair $(u, w) \in \mathrm{gra}(\partial \hc)$, the inequality $\langle u^\star - u ,\K x^\infty - w\rangle \ge 0$ is satisfied \citep[Definition 20.20]{bau17}. Let $(u, w) \in \mathrm{gra}(\partial \hc)$. For every $k\geq 1$, monotonicity yields $\langle v^k - u,\tilde{\nabla}\hc(v^k) - w \rangle \ge 0$, which we can expand as
\begin{equation}
0\leq \langle u^\star - u,\K z^k - w \rangle + e_k,\quad \text{where} \quad e_k\eqdef \langle v^k - u^\star, \K z^k - w  \rangle - \langle v^k - u, \K z^k  -\tilde{\nabla}\hc(v^k)\rangle.\label{eqin3}
\end{equation}
Since $(\K z^t - w)_{t\ge 0}$ is bounded, Assumptions $\mathrm{(ii)}$ and $\mathrm{(v)}$ imply that $e_k \rightarrow 0$. 
By Lemma~\ref{lemavg}, for every $t\geq 1$, $x^t = \sum_{k=1}^t w_{k,t} z^k$ with the nonnegative weights $w_{k,t}$ defined in \eqref{eqavg}, which satisfy $\sum_{k=1}^t w_{k,t} = 1$ (here, $w_{0,t}=0$, since $a_1=1$). 
Multiplying the inequality \eqref{eqin3} by $w_{k,t}$ and summing over $k=1,\ldots,t$ yields
\begin{equation}
\langle u^\star - u,\K x^t - w \rangle + \sum_{k=1}^t w_{k,t} e_k \ge 0.\label{eqin2}
\end{equation}
By Lemma~\ref{lemavg}$\mathrm{(ii)}$, $\sum_{k=1}^t w_{k,t} e_k \to 0$. 
Let $(t_j)_{j \ge 0}$ be the subsequence such that $x^{t_j} \rightharpoonup x^\infty$. Because $\K$ is a bounded linear operator, it is weakly continuous, ensuring $\K x^{t_j} \rightharpoonup \K x^\infty$. Passing to the limit as $j \to \infty$ in \eqref{eqin2}, we obtain $\langle u^\star - u,\K x^\infty - w \rangle \ge 0$. This confirms that $\K x^\infty \in \partial \hc(u^\star)$. 
Therefore, $(x^\infty,u^\star)$ is a solution to \eqref{eq2}, so that $x^\infty$ is a primal solution to \eqref{eq1c}.

\textbf{Step 2: Uniqueness of the weak cluster point.} We now assume $\mathrm{(vi)}$--$\mathrm{(vii)}$ and let $\hat{x}^\star$ and $\check{x}^\star$ be two weak cluster points of $(x^t)_{t\geq 0}$. By Step 1, both are primal solutions, so that $(\hat{x}^\star,u^\star)$ and $(\check{x}^\star,u^\star)$ are solutions to \eqref{eq2}. By Assumption $\mathrm{(vi)}$, there exist nonnegative, nonincreasing Lyapunov sequences $(\mathcal{E}_{\hat{x}^\star}^t)_{t\geq 0}$ and $(\mathcal{E}_{\check{x}^\star}^t)_{t\geq 0}$, which therefore both converge. Evaluating their difference, using  \eqref{eqgapw} and applying Assumption $\mathrm{(vii)}$ yields 
\begin{align*}
2\gamma\big(\mathcal{E}_{\hat{x}^\star}^t - \mathcal{E}_{\check{x}^\star}^t\big) &=  \|z^t-\hat{x}^\star\|^2 - \|z^t-\check{x}^\star\|^2 + 2\gamma a_t^2 \big( \mathcal{G}_{\hat{x}^\star,u^\star}(x^t, u^t) - \mathcal{G}_{\check{x}^\star,u^\star}(x^t, u^t) \big)\\
&=  \|z^t-\hat{x}^\star\|^2 - \|z^t-\check{x}^\star\|^2+ 2\gamma a_t^2 \langle \K\check{x}^\star - \K\hat{x}^\star, u^t - u^\star \rangle \\
&=\|z^t-\hat{x}^\star\|^2 - \|z^t-\check{x}^\star\|^2.
\end{align*}
Consequently, the sequence $\big(\sqn{z^t-\hat{x}^\star}-\sqn{z^t-\check{x}^\star}\big)_{t\geq 0}$ converges. Since $(z^t)_{t\geq 0}$ is bounded by Assumption $\mathrm{(iii)}$ and $\sum_{t\geq 1}\frac{1}{a_t}=\infty$ by Assumption $\mathrm{(i)}$, Lemma~\ref{lemavg}$\mathrm{(iii)}$ applies: $(x^t)_{t\geq 0}$ converges weakly to some $x^\star$, which is a primal solution by Step 1. Finally, because $\|y^t-x^t\|\to 0$ by $\mathrm{(a)}$, the sequence $(y^t)_{t\geq 0}$ also converges weakly to $x^\star$.
\end{proof}

Regarding Assumption $\mathrm{(vii)}$ of Theorem \ref{theoconva}, we note that $\K \hat{x}^\star$ and $\K \check{x}^\star$ both belong to $\partial \hc(u^\star)$, so that a sufficient condition for $\K\hat{x}^\star=\K \check{x}^\star$ is simply that $\partial \hc(u^\star)$ is a singleton.
This occurs, for instance, if $\h$ is strictly convex on $\mathcal{U}$. In particular, it holds in the linearly constrained setting of Section \ref{seclc}, where $\partial \hc(u^\star)=\{b\}$. We formally state this convergence result as follows:

\begin{theorem}[Point Convergence of \algn{APAPC} for Linearly Constrained Problems]\label{theolcp}
Consider the linearly constrained setting of Section \ref{seclc},
 with $\lambda^+_{\min}(\K\K^*) >0$ and $\mu_\g=0$. Let $(x^\star,u^\star)$ be a solution of \eqref{eq2lc}. In \algn{APAPC}, define the Lyapunov function 
\begin{align*}
\mathcal{E}^t_{x^\star}&\coloneqq \frac{1}{2\gamma}\sqn{z^{t}-x^\star}+
\frac{2a_{t}^2+\min\big(\frac{1}{2},\frac{1}{2a_{t}\gamma L_\ff}\big)
\gamma \tau a_{t}^2 \lambda^+_{\min}(\K\K^*)}{4\tau}
\sqn{v^{t}-u^{\star}}  \\
&\quad -\frac{a_{t}^2\gamma}{2}\sqn{\K^*(v^{t}-u^{\star})}+a_{t}^2 
\mathcal{G}_{x^\star,u^\star}(x^t,u^t)+\frac{a_{t}\tau}{2}\sqn{\K x^{t}-b},\quad\forall t\geq 0.
\end{align*}
Suppose that $u^0=v^0\in\mathrm{ran}(\K)$, $\gamma\leq \frac{1}{2 L_\ff}$, $\tau=\frac{\nu}{\gamma\|\K\|^2}$ for some $\nu\in (0,1]$, $a_0=a_1=1$, and
\begin{equation*}
a_{t+1}\coloneqq \min\left(a_t{\textstyle \sqrt{1+\frac{\nu\lambda^+_{\min}(\K\K^*)}{4 \|\K\|^2}}},{\textstyle \sqrt{a_t^2+a_{t}\frac{\nu\lambda^+_{\min}(\K\K^*)}{4\gamma L_\ff \|\K\|^2}},\frac{1+\sqrt{1+4a_t^2}}{2}}
\right),\quad\forall t\geq 1.
\end{equation*}
Then, by Theorem \ref{theolc}, the sequence
 $(\mathcal{E}^{t}_{x^\star})_{t\ge 0}$ is nonincreasing (which implies that $(z^t)_{t \ge 0}$ is bounded), $\mathcal{G}_{x^\star,u^\star}(x^t,u^t) =\mathcal{O}(1/t^2)$, $a_t=\Theta(t)$, and $\sqn{v^{t}-u^{\star}}=\mathcal{O}(1/t)$. Moreover, let $\widetilde{\mathcal{E}}^{t}_{x^\star}\coloneqq \mathcal{E}^{t}_{x^\star}-\frac{a_{t}\tau}{2}\sqn{\K x^{t}-b}\geq 0$; the inequality \eqref{eq99} in the proof of Theorem \ref{theolc}, combined with the term-by-term comparison made there, 
yields
\begin{equation*}
\widetilde{\mathcal{E}}^{t+1}_{x^\star}\leq \widetilde{\mathcal{E}}^{t}_{x^\star}-\frac{\tau}{2}\sqn{\K z^{t+1}-b},\quad\forall t\geq 0,\qquad\mbox{so that}\qquad \sum_{t\geq 0}\sqn{\K z^{t+1}-b}\leq \frac{2}{\tau}\widetilde{\mathcal{E}}^{0}_{x^\star}<+\infty
\end{equation*}
and $\big\|\K z^{t}-\tilde{\nabla}\hc (v^t)\big\|=\|\K z^{t}-b\| \to 0$.
 Because $\K x^\star=b$ for any primal solution $x^\star$, Assumptions $\mathrm{(i)}$--$\mathrm{(vii)}$ of Theorem \ref{theoconva} are all satisfied, with $\mathcal{U}$ replaced by $\widetilde{\mathcal{U}}=\mathrm{ran}(\K)$, on which the dual solution $u^\star$ is unique. Consequently, there exists a solution $x^\star$ to \eqref{eq1lc} such that $(x^t)_{t\geq 0}$ and $(y^t)_{t\geq 0}$ both converge weakly to $x^\star$. 
\end{theorem}

\begin{remark}\label{remconvother}
Theorem \ref{theoconva} also applies in the two other settings of Sections \ref{sechs} and \ref{secinj}, with $\mu_\g=0$. Under the conditions of Theorem \ref{corcase1a} or Theorem \ref{theoinj}, respectively, the Lyapunov function is nonincreasing and has the form required in Assumption $\mathrm{(vi)}$, which yields Assumptions $\mathrm{(iii)}$ and $\mathrm{(iv)}$; Assumption $\mathrm{(i)}$ holds for the proposed sequences $(a_t)_{t\geq 0}$; Assumption $\mathrm{(ii)}$ follows from $\mathcal{E}^{t}_{x^\star}\leq \mathcal{E}^{0}_{x^\star}$, which implies $\sqn{v^t-u^\star}\leq \frac{2\mathcal{E}^{0}_{x^\star}}{a_t\mu_\hc}\rightarrow 0$ in the setting of Section \ref{sechs}, and is part of the conclusions of Theorem \ref{theoinj}; and Assumption $\mathrm{(v)}$ follows, exactly as in Theorem \ref{theolcp}, by keeping the term $- \frac{\tau}{2}\sqn{ \K z^{t+1}-\tilde{\nabla}\hc(v^{t+1})}$ in \eqref{eqinea1} and summing over $t$. Thus, in all three settings, the weak convergence of $(x^t)_{t\geq 0}$ to a primal solution hinges on Assumption $\mathrm{(vii)}$ alone.
\end{remark}

\section{Conclusion}

We introduced the Accelerated Proximal Alternating Predictor--Corrector algorithm (\algn{APAPC}), a fully split primal--dual algorithm that embeds Nesterov momentum into the forward--backward structure of \algn{PAPC}. Its analysis rests on one principle, namely that Nesterov acceleration exploits the strong convexity of the dual problem, and on one Lyapunov inequality, from which optimal $\mathcal{O}(1/t^2)$ rates and accelerated linear rates follow in three regimes of dual strong convexity, by the choice of the momentum sequence alone. The $\mathcal{O}(1/t^2)$ rates for linearly constrained problems and when $\K^*$ is bounded below, without strong convexity of the primal problem, are new for fully split algorithms. The same framework, combined with recent advances on the point convergence of accelerated gradient methods, yields the weak convergence of the iterates of \algn{APGD} and \algn{APAPC}.

\section*{Acknowledgements}
This work was supported by funding from King Abdullah University of Science and Technology (KAUST):\\
i) KAUST Baseline Research Scheme,\\
ii) Center of Excellence for Generative AI (award no.\ 5940),\\
iii) Competitive Research Grant (CRG) Program (award no.\ 6460).

\bibliographystyle{abbrvnat}
\bibliography{IEEEabrv,biblio}

\begin{thebibliography}{62}
\providecommand{\natexlab}[1]{#1}
\providecommand{\url}[1]{\texttt{#1}}
\expandafter\ifx\csname urlstyle\endcsname\relax
  \providecommand{\doi}[1]{doi: #1}\else
  \providecommand{\doi}{doi: \begingroup \urlstyle{rm}\Url}\fi

\bibitem[Alghunaim et~al.(2021)Alghunaim, Ryu, Yuan, and Sayed]{alg21}
S.~A. Alghunaim, E.~K. Ryu, K.~Yuan, and A.~H. Sayed.
\newblock Decentralized proximal gradient algorithms with linear convergence
  rates.
\newblock \emph{IEEE Trans. Automatic Control}, 66\penalty0 (6):\penalty0
  2787--2794, June 2021.

\bibitem[Allen-Zhu and Orecchia(2017)]{all14}
Z.~Allen-Zhu and L.~Orecchia.
\newblock Linear coupling: An ultimate unification of gradient and mirror
  descent.
\newblock In \emph{{Proc.\ of} Innovations in Theoretical Computer Science
  Conf. (ITCS)}, 2017.

\bibitem[Auslender and Teboulle(2006)]{aus06}
A.~Auslender and M.~Teboulle.
\newblock Interior gradient and proximal methods for convex and conic
  optimization.
\newblock \emph{SIAM J. Optim.}, 16\penalty0 (3):\penalty0 697--725, 2006.

\bibitem[Bauschke and Combettes(2017)]{bau17}
H.~H. Bauschke and P.~L. Combettes.
\newblock \emph{Convex Analysis and Monotone Operator Theory in Hilbert
  Spaces}.
\newblock Springer, Cham, 2nd edition, 2017.

\bibitem[Bauschke and Moursi(2026)]{bau26}
H.~H. Bauschke and W.~M. Moursi.
\newblock Understanding {FISTA}'s weak convergence: A step-by-step introduction
  to the 2025 milestone.
\newblock preprint arXiv:2601.15398, 2026.

\bibitem[Beck(2017)]{bec17}
A.~Beck.
\newblock \emph{First-Order Methods in Optimization}.
\newblock MOS-SIAM Series on Optimization. SIAM, 2017.

\bibitem[Beck and Teboulle(2009)]{bec092}
A.~Beck and M.~Teboulle.
\newblock A fast iterative shrinkage-thresholding algorithm for linear inverse
  problems.
\newblock \emph{SIAM J. Imaging Sci.}, 2\penalty0 (1):\penalty0 183--202, 2009.

\bibitem[Bo{\c{t}} et~al.(2014)Bo{\c{t}}, Csetnek, and Hendrich]{bot14}
R.~I. Bo{\c{t}}, E.~R. Csetnek, and C.~Hendrich.
\newblock Recent developments on primal--dual splitting methods with
  applications to convex minimization.
\newblock In P.~M. Pardalos and T.~M. Rassias, editors, \emph{Mathematics
  Without Boundaries: Surveys in Interdisciplinary Research}, pages 57--99.
  Springer New York, 2014.

\bibitem[Bo{\c{t}} et~al.(2023)Bo{\c{t}}, Csetnek, and Nguyen]{bot23}
R.~I. Bo{\c{t}}, E.~R. Csetnek, and D.-K. Nguyen.
\newblock Fast augmented {L}agrangian method in the convex regime with
  convergence guarantees for the iterates.
\newblock \emph{Math. Program.}, 200\penalty0 (1):\penalty0 147--197, 2023.

\bibitem[Bo{\c{t}} et~al.(2025{\natexlab{a}})Bo{\c{t}}, Fadili, and
  Nguyen]{bot252}
R.~I. Bo{\c{t}}, J.~Fadili, and D.-K. Nguyen.
\newblock The iterates of {Nesterov’s} accelerated algorithm converge in the
  critical regimes.
\newblock preprint arXiv:2510.22715, Oct. 2025{\natexlab{a}}.

\bibitem[Bo{\c{t}} et~al.(2025{\natexlab{b}})Bo{\c{t}}, Nguyen, and
  Zong]{bot253}
R.~I. Bo{\c{t}}, D.-K. Nguyen, and C.~Zong.
\newblock Fast reflected forward-backward algorithm: achieving fast convergence
  rates for convex optimization with linear cone constraints.
\newblock \emph{J. Sci. Comput.}, 105\penalty0 (3):\penalty0 73,
  2025{\natexlab{b}}.

\bibitem[Bo{\c{t}} et~al.(2026)Bo{\c{t}}, Chenchene, Csetnek, and
  Hulett]{bot25}
R.~I. Bo{\c{t}}, E.~Chenchene, E.~R. Csetnek, and D.~A. Hulett.
\newblock Accelerating diagonal methods for bilevel optimization: Unified
  convergence via continuous-time dynamics.
\newblock \emph{Math. Program.}, 2026.

\bibitem[Bottou et~al.(2018)Bottou, Curtis, and Nocedal]{bot18}
L.~Bottou, F.~E. Curtis, and J.~Nocedal.
\newblock Optimization methods for large-scale machine learning.
\newblock \emph{SIAM Review}, 60\penalty0 (2):\penalty0 223--311, 2018.

\bibitem[Bubeck(2015)]{bub15}
S.~Bubeck.
\newblock Convex optimization: {A}lgorithms and complexity.
\newblock \emph{Found. Trends Mach. Learn.}, 8\penalty0 (3--4):\penalty0
  231--357, 2015.

\bibitem[Chambolle and Dossal(2015)]{dos15}
A.~Chambolle and C.~Dossal.
\newblock On the convergence of the iterates of the {``Fast Iterative
  Shrinkage/Thresholding Algorithm''}.
\newblock \emph{J. Optim. Theory Appl.}, 166\penalty0 (3):\penalty0 968--982,
  2015.

\bibitem[Chambolle and Pock(2011)]{cha11a}
A.~Chambolle and T.~Pock.
\newblock A first-order primal-dual algorithm for convex problems with
  applications to imaging.
\newblock \emph{J. Math. Imaging Vision}, 40\penalty0 (1):\penalty0 120--145,
  May 2011.

\bibitem[Chambolle and Pock(2016{\natexlab{a}})]{cha16}
A.~Chambolle and T.~Pock.
\newblock An introduction to continuous optimization for imaging.
\newblock \emph{Acta Numerica}, 25:\penalty0 161--319, 2016{\natexlab{a}}.

\bibitem[Chambolle and Pock(2016{\natexlab{b}})]{cha162}
A.~Chambolle and T.~Pock.
\newblock On the ergodic convergence rates of a first-order {primal--dual}
  algorithm.
\newblock \emph{Math. Program.}, 159\penalty0 (1--2):\penalty0 253--287, Sept.
  2016{\natexlab{b}}.

\bibitem[Chen et~al.(2013)Chen, Huang, and Zhang]{che13}
P.~Chen, J.~Huang, and X.~Zhang.
\newblock A {primal--dual} fixed point algorithm for convex separable
  minimization with applications to image restoration.
\newblock \emph{Inverse Problems}, 29\penalty0 (2):\penalty0 025011, 2013.

\bibitem[Chen et~al.(2014)Chen, Lan, and Ouyang]{che142}
Y.~Chen, G.~Lan, and Y.~Ouyang.
\newblock Optimal primal-dual methods for a class of saddle point problems.
\newblock \emph{SIAM J. Optim.}, 24\penalty0 (4):\penalty0 1779--1814, 2014.

\bibitem[Combettes(2024)]{com24}
P.~L. Combettes.
\newblock The geometry of monotone operator splitting methods.
\newblock \emph{Acta Numerica}, 33:\penalty0 487--632, 2024.

\bibitem[Combettes and Pesquet(2011)]{com10}
P.~L. Combettes and J.-C. Pesquet.
\newblock Proximal splitting methods in signal processing.
\newblock In H.~H. Bauschke, R.~Burachik, P.~L. Combettes, V.~Elser, D.~R.
  Luke, and H.~Wolkowicz, editors, \emph{Fixed-Point Algorithms for Inverse
  Problems in Science and Engineering}, pages 185--212. Springer-Verlag, New
  York, 2011.

\bibitem[Combettes and Pesquet(2021)]{com21}
P.~L. Combettes and J.-C. Pesquet.
\newblock Fixed point strategies in data science.
\newblock \emph{{IEEE} Trans. Signal Process.}, 69:\penalty0 3878--3905, 2021.

\bibitem[Combettes et~al.(2014)Combettes, Condat, Pesquet, and {V\~u}]{com14}
P.~L. Combettes, L.~Condat, J.-C. Pesquet, and B.~C. {V\~u}.
\newblock A forward--backward view of some primal--dual optimization methods in
  image recovery.
\newblock In \emph{{Proc.\ of} IEEE ICIP}, Paris, France, Oct. 2014.

\bibitem[Condat(2013)]{con13}
L.~Condat.
\newblock A primal-dual splitting method for convex optimization involving
  {L}ipschitzian, proximable and linear composite terms.
\newblock \emph{J. Optim. Theory Appl.}, 158\penalty0 (2):\penalty0 460--479,
  2013.

\bibitem[Condat and Richt{\'a}rik(2023)]{con22rp}
L.~Condat and P.~Richt{\'a}rik.
\newblock {RandProx}: {P}rimal-dual optimization algorithms with randomized
  proximal updates.
\newblock In \emph{{Proc.\ of} Int. Conf. Learning Representations (ICLR)},
  2023.

\bibitem[Condat et~al.(2022)Condat, Malinovsky, and Richt{\'a}rik]{con22}
L.~Condat, G.~Malinovsky, and P.~Richt{\'a}rik.
\newblock Distributed proximal splitting algorithms with rates and
  acceleration.
\newblock \emph{Frontiers in Signal Processing}, 1, Jan. 2022.

\bibitem[Condat et~al.(2023)Condat, Kitahara, Contreras, and
  Hirabayashi]{con23}
L.~Condat, D.~Kitahara, A.~Contreras, and A.~Hirabayashi.
\newblock Proximal splitting algorithms for convex optimization: A tour of
  recent advances, with new twists.
\newblock \emph{SIAM Review}, 65\penalty0 (2):\penalty0 375--435, 2023.

\bibitem[Condat et~al.(2025)Condat, Gasanov, and {Richt\'arik}]{con25smpm}
L.~Condat, E.~Gasanov, and P.~{Richt\'arik}.
\newblock The stochastic multi-proximal method for nonsmooth optimization.
\newblock preprint arXiv:2505.12409, 2025.

\bibitem[Davis and Yin(2017)]{dav17}
D.~Davis and W.~Yin.
\newblock A three-operator splitting scheme and its optimization applications.
\newblock \emph{Set-Val. Var. Anal.}, 25\penalty0 (4):\penalty0 829--858, 2017.

\bibitem[Dirren et~al.(2025)Dirren, Bianchi, Grontas, Lygeros, and
  {D\"orfler}]{dir25}
C.~Dirren, M.~Bianchi, P.~D. Grontas, J.~Lygeros, and F.~{D\"orfler}.
\newblock Contractivity and linear convergence in bilinear saddle-point
  problems: An operator-theoretic approach.
\newblock In \emph{{Proc.\ of Int.\ Conf.\ Artificial} Intelligence and
  Statistics (AISTATS)}, volume PMLR 258, pages 1954--1962, 2025.

\bibitem[Driggs et~al.(2024)Driggs, Ehrhardt, {Sch\"onlieb}, and Tang]{dri24}
D.~Driggs, M.~J. Ehrhardt, C.-B. {Sch\"onlieb}, and J.~Tang.
\newblock Practical acceleration of the {Condat--V\~u} algorithm.
\newblock \emph{SIAM J. Imaging Sci.}, 17\penalty0 (4):\penalty0 2076--2109,
  2024.

\bibitem[Drori et~al.(2015)Drori, Sabach, and Teboulle]{dro15}
Y.~Drori, S.~Sabach, and M.~Teboulle.
\newblock A simple algorithm for a class of nonsmooth convex concave
  saddle-point problems.
\newblock \emph{Oper. Res. Lett.}, 43\penalty0 (2):\penalty0 209--214, 2015.

\bibitem[Glowinski et~al.(2016)Glowinski, Osher, and Yin]{glo16}
R.~Glowinski, S.~J. Osher, and W.~Yin, editors.
\newblock \emph{Splitting Methods in Communication, Imaging, Science, and
  Engineering}.
\newblock Springer International Publishing, 2016.

\bibitem[He(2024)]{he24}
X.~He.
\newblock Accelerated primal-dual methods with adaptive parameters for
  composite convex optimization with linear constraints.
\newblock \emph{Appl. Numer. Math.}, 203:\penalty0 129--143, 2024.

\bibitem[He et~al.(2025)He, Huang, and Fang]{he25}
X.~He, N.-J. Huang, and Y.-P. Fang.
\newblock Non-ergodic convergence rate of an inertial accelerated primal-dual
  algorithm for saddle point problems.
\newblock \emph{Commun. Nonlinear Sci. Numer. Simul.}, 140:\penalty0 108289,
  2025.

\bibitem[Jang and Ryu(2025)]{jan252}
U.~Jang and E.~K. Ryu.
\newblock Point convergence of {Nesterov’s} accelerated gradient method: an
  {AI}-assisted proof.
\newblock preprint arXiv:2510.23513, Oct. 2025.

\bibitem[Jang et~al.(2026)Jang, {Das Gupta}, and Ryu]{jan25}
U.~Jang, S.~{Das Gupta}, and E.~K. Ryu.
\newblock Computer-assisted design of accelerated composite optimization
  methods: {OptISTA}.
\newblock \emph{Math. Program.}, 219\penalty0 (1):\penalty0 391--499, 2026.

\bibitem[Knopp(1951)]{kno51}
K.~Knopp.
\newblock \emph{Theory and Application of Infinite Series}.
\newblock Blackie {\&} Son, 2nd edition, 1951.

\bibitem[Kovalev et~al.(2020)Kovalev, Salim, and {Richt\'arik}]{kov20}
D.~Kovalev, A.~Salim, and P.~{Richt\'arik}.
\newblock Optimal and practical algorithms for smooth and strongly convex
  decentralized optimization.
\newblock In \emph{{Proc.\ of} Conf. on Neural Information Processing Systems
  (NeurIPS)}, 2020.

\bibitem[Lan and Ouyang(2022)]{lan22}
G.~Lan and Y.~Ouyang.
\newblock Accelerated gradient sliding for structured convex optimization.
\newblock \emph{Comput. Optim. Appl.}, 82\penalty0 (2):\penalty0 361--394,
  2022.

\bibitem[Lee et~al.(2025)Lee, Yi, and Ryu]{lee25}
J.~Lee, S.~Yi, and E.~K. Ryu.
\newblock Convergence analyses of {Davis--Yin} splitting via scaled relative
  graphs.
\newblock \emph{SIAM J. Optim.}, 35\penalty0 (1):\penalty0 270--301, 2025.

\bibitem[Loris and Verhoeven(2011)]{lor11}
I.~Loris and C.~Verhoeven.
\newblock On a generalization of the iterative soft-thresholding algorithm for
  the case of non-separable penalty.
\newblock \emph{Inverse Problems}, 27\penalty0 (12):\penalty0 125007, 2011.

\bibitem[Luo(2024)]{luo24}
H.~Luo.
\newblock Accelerated primal-dual proximal gradient splitting methods for
  convex-concave saddle-point problems.
\newblock preprint arXiv:2407.20195, 2024.

\bibitem[Luo(2026)]{luo26}
H.~Luo.
\newblock Accelerated primal-dual methods for linearly constrained convex
  optimization problems.
\newblock \emph{J. Global Optim.}, 95\penalty0 (2):\penalty0 405--437, 2026.

\bibitem[Nesterov(2004)]{nes04}
Y.~Nesterov.
\newblock \emph{Introductory lectures on convex optimization: a basic course}.
\newblock Kluwer Academic Publishers, 2004.

\bibitem[Nesterov(2013)]{nes13}
Y.~Nesterov.
\newblock Gradient methods for minimizing composite functions.
\newblock \emph{Math. Program.}, 140\penalty0 (1):\penalty0 125--161, 2013.

\bibitem[{O'Connor} and Vandenberghe(2020)]{oco20}
D.~{O'Connor} and L.~Vandenberghe.
\newblock On the equivalence of the primal-dual hybrid gradient method and
  {Douglas--Rachford} splitting.
\newblock \emph{Math. Program.}, 179\penalty0 (1):\penalty0 85--108, 2020.

\bibitem[Ouyang and Xu(2021)]{ouy21}
Y.~Ouyang and Y.~Xu.
\newblock Lower complexity bounds of first-order methods for convex-concave
  bilinear saddle-point problems.
\newblock \emph{Math. Program.}, 185\penalty0 (1--2):\penalty0 1--35, 2021.

\bibitem[Ouyang et~al.(2015)Ouyang, Chen, Lan, and Pasiliao]{ouy15}
Y.~Ouyang, Y.~Chen, G.~Lan, and E.~Pasiliao, Jr.
\newblock An accelerated linearized alternating direction method of
  multipliers.
\newblock \emph{SIAM J. Imaging Sci.}, 8\penalty0 (1):\penalty0 644--681, 2015.

\bibitem[Parikh and Boyd(2014)]{par14}
N.~Parikh and S.~Boyd.
\newblock Proximal algorithms.
\newblock \emph{Foundations and Trends in Optimization}, 1\penalty0
  (3):\penalty0 127--239, 2014.

\bibitem[Salim et~al.(2022{\natexlab{a}})Salim, Condat, Kovalev, and
  {Richt\'arik}]{sal22}
A.~Salim, L.~Condat, D.~Kovalev, and P.~{Richt\'arik}.
\newblock An optimal algorithm for strongly convex minimization under affine
  constraints.
\newblock In \emph{{Proc.\ of Int.\ Conf.\ Artificial} Intelligence and
  Statistics (AISTATS)}, volume PMLR 151, pages 4482--4498, 2022{\natexlab{a}}.

\bibitem[Salim et~al.(2022{\natexlab{b}})Salim, Condat, Mishchenko, and
  {Richt\'arik}]{sal20}
A.~Salim, L.~Condat, K.~Mishchenko, and P.~{Richt\'arik}.
\newblock Dualize, split, randomize: Toward fast nonsmooth optimization
  algorithms.
\newblock \emph{J. Optim. Theory Appl.}, 195\penalty0 (1):\penalty0 102--130,
  2022{\natexlab{b}}.

\bibitem[Taylor et~al.(2017)Taylor, Hendrickx, and Glineur]{tay17}
A.~B. Taylor, J.~M. Hendrickx, and F.~Glineur.
\newblock Smooth strongly convex interpolation and exact worst-case performance
  of first-order methods.
\newblock \emph{Math. Program.}, 161\penalty0 (1--2):\penalty0 307--345, 2017.

\bibitem[Taylor et~al.(2018)Taylor, Hendrickx, and Glineur]{tay18}
A.~B. Taylor, J.~M. Hendrickx, and F.~Glineur.
\newblock Exact worst-case convergence rates of the proximal gradient method
  for composite convex minimization.
\newblock \emph{J. Optim. Theory Appl.}, 178\penalty0 (2):\penalty0 455--476,
  2018.

\bibitem[Tseng(2008)]{tse08}
P.~Tseng.
\newblock On accelerated proximal gradient methods for convex--concave
  optimization.
\newblock preprint, 2008.

\bibitem[Ushiyama(2025)]{ush25}
K.~Ushiyama.
\newblock A {$\sqrt{2}$}-accelerated {FISTA} for composite strongly convex
  problems.
\newblock preprint arXiv:2509.09295, Sept. 2025.

\bibitem[V{\~u}(2013)]{vu13}
B.~C. V{\~u}.
\newblock A splitting algorithm for dual monotone inclusions involving
  cocoercive operators.
\newblock \emph{Adv. Comput. Math.}, 38\penalty0 (3):\penalty0 667--681, Apr.
  2013.

\bibitem[Yan(2018)]{yan18}
M.~Yan.
\newblock A new primal-dual algorithm for minimizing the sum of three functions
  with a linear operator.
\newblock \emph{J. Sci. Comput.}, 76\penalty0 (3):\penalty0 1698--1717, Sept.
  2018.

\bibitem[Zhao(2022)]{zha22}
R.~Zhao.
\newblock Accelerated stochastic algorithms for convex-concave saddle-point
  problems.
\newblock \emph{Math. Oper. Res.}, 47\penalty0 (2):\penalty0 1443--1473, 2022.

\bibitem[Zhao et~al.(2019)Zhao, Haskell, and Tan]{zha19}
R.~Zhao, W.~B. Haskell, and V.~Y.~F. Tan.
\newblock An optimal algorithm for stochastic three-composite optimization.
\newblock In \emph{{Proc.\ of} International Conference on Artificial
  Intelligence and Statistics (AISTATS)}, volume PMLR 89, pages 428--437, 2019.

\bibitem[Zhu(2025)]{zhu25}
Y.-N. Zhu.
\newblock Accelerated primal dual fixed point algorithm.
\newblock preprint arXiv:2511.00385, 2025.

\end{thebibliography}

\end{document}